\documentclass[12pt,epsfig,amsfonts]{amsart}

\usepackage{hyperref, soul, cancel}
\usepackage{graphicx}
\usepackage{tikz-cd}
\usepackage{euscript}
\usepackage{amsmath,amssymb,amscd,epsfig}
\usepackage{mathrsfs}
\usepackage{todonotes}

\usepackage{xcolor}

\newcommand{\R}{\mathbb R}

\newcommand{\T}{\mathbb T}

\newcommand{\N}{\mathbb N}
\newcommand{\Px}{\mathcal{P}}

\newcommand{\Z}{\mathbb Z}
\newcommand{\holder}{H\"{o}lder~}

\newtheorem{thm}{Theorem}[section]
\newtheorem{prop}[thm]{Proposition}
\newtheorem{lem}[thm]{Lemma}

\def \a{\alpha}

\begin{document}

\title[Polynomial decay of correlations]{Area preserving surface diffeomorphisms with polynomial decay of correlations are ubiquitous}

\author{Y. Pesin}
\address{Department of Mathematics, Pennsylvania State
University, University Park, PA 16802, USA}
\email{pesin@math.psu.edu}
\author{S. Senti}
\address{Instituto de Matematica, Universidade Federal do Rio de Janeiro, C.P. 68 530, CEP 21945-970, R.J., Brazil}
\email{senti@im.ufrj.br}
\author{F. Shahidi}
\address{Department of Mathematics,
Pennsylvnia State University, University Park, PA, 16802, USA}
\email{fus144@psu.edu}

\date{\today}

\thanks{} 


\begin{abstract} 
We show that any smooth compact connected and oriented surface admits an area preserving $C^{1+\beta}$ diffeomorphism with non-zero Lyapunov exponents which is Bernoulli and has polynomial decay of correlations. We establish both upper and lower polynomial bounds on correlations. In addition, we show that this diffeomorphism satisfies the Central Limit Theorem and has the Large Deviation Property. Finally, we show that the diffeomorphism we constructed possesses a unique hyperbolic  Bernoulli measure of maximal entropy with respect to which it has exponential decay of correlations.
\end{abstract}
\maketitle

\section{Introduction}

A classical problem in smooth dynamics known as the \emph{smooth realization problem} asks whether there is a diffeomorphism $f$ of a compact smooth manifold $M$ which has a prescribed collection of ergodic properties with respect to a \emph{natural} invariant measure $\mu$ such as the Riemannian volume (or a more general smooth measure, i.e., a measure that is equivalent to volume). Other interesting measures to consider include the measure of maximal entropy. A yet more interesting but substantially more difficult version of the smooth realization problem is to construct a volume preserving diffeomorphism $f$ with prescribed ergodic properties on any given smooth manifold $M$. Starting with the basic ergodic property -- ergodicity -- Anosov and Katok \cite{AnoKat} constructed an example of a volume preserving \emph{ergodic} $C^{\infty}$ map with some additional metric properties. Katok \cite{katan} gave an example of area preserving $C^{\infty}$ diffeomorphism with non-zero Lyapunov exponents on any surface which is \emph{Bernoulli} (see the definitions in the next section). Later Brin, Feldman, and Katok \cite{BrinFelKat} and then Brin \cite{Brin} extended this result by constructing a volume preserving $C^{\infty}$ diffeomorphism, which is Bernoulli, on any Riemannian manifold of dimension $\ge 5$. In this example the map has all but one non-zero Lyapunov exponents. Finally, Dolgopyat and Pesin \cite{DolPes} constructed a volume preserving $C^{\infty}$ Bernoulli diffeomorphism with \emph{non-zero Lyapunov exponents} on any Riemannian manifold of dimension $\ge 2$.

It is natural to ask if a compact smooth manifold admits a volume preserving Bernoulli diffeomorphism with non-zero Lyapunov exponents that enjoys other important statistical properties such as exponential or polynomial decay of correlations (that is rate of mixing), the Central Limit Theorem, and the Large Deviations property (all three with respect to a \emph{natural} class of observables, e.g., functions which are H\"older continuous).
 
In one dimensional dynamics the famous Mauneville-Pomeau map \cite{ManPom80} (with some modifications) provide some examples of a map with an indifferent fixed point  preserving a measure which is absolutely continuous with respect to the (one-dimensional) Lebesgue measure. With respect to this measure the decay of correlations is polynomial, the Central Limit Theorem is satisfied as is the Large Deviations property (with respect to the class of H\"older continuous observables; see \cite{Gou,Hu,LSV} and Section 4 below). 

In the two dimensional case Liverani and Martens \cite{LivMer}  constructed an example of an area preserving $C^{\infty}$ diffeomorphism of the $2$-torus with non-zero Lyapunov exponents which has polynomial decay of correlations with respect to the class of smooth observables\footnote{More precisely, it is shown in \cite{LivMer} that the map admits a polynomial upper bound.}. In the present paper we show that any surface admits an area preserving $C^{1+\beta}$ diffeomorphism with non-zero Lyapunov exponents which is Bernoulli and has polynomial decay of correlations -- more precisely, it allows polynomial lower and upper bounds. It also satisfies the Central Limit Theorem and has the Large Deviation property. 

Interestingly enough the map we construct also has the unique measure of maximal entropy with exponential decay of correlations. Thus we show that any surface allows a $C^{1+\beta}$ diffeomorphism with the unique measure of maximal entropy with respect to which it has non-zero Lyapunov exponents, is Bernoulli, has exponential decay of correlations, and satisfies the Central Limit Theorem.

Our proof follows Katok, \cite{katan}. First, we construct a map $f_{T^2}$ with the desired properties on the $2$-torus. Starting with a linear automorphism of the torus which has $4$ fixed points, this is done by slowing down trajectories in sufficiently small neighborhoods of these points and then \emph{correcting} this map so that the resulting diffeomorphism preserves area. This yields a $C^{2+2\kappa}$ area preserving diffeomorphism with non-zero Lyapunov exponents which is Bernoulli. 

To study the decay of correlations we represent $f_{T^2}$ as a Young diffeomorphism and study the symbolic map on the corresponding Young tower (see Section 5). This map preserves the measure that is the lift of the area (see \cite{SZ, Youngannals}). The decay of correlations of this symbolic map has been studied extensively (see for example, \cite{Gou, mel, sarigsub, SZ, Youngannals, Youngdecay}) and is tied to the decay of the tail of the return time. Thus to establish polynomial upper and lower bounds on the decay of correlations we need to obtain both upper and lower bounds on decay of the tail. This is done in Sections 6-8 which constitute the technically most difficult part of the work. 

Our next step is to carry over the map of the torus to a map on a given surface. To achieve this we follow the approach in \cite{katan} and obtain a $C^{2+2\kappa}$ area preserving Bernoulli diffeomorphism $f_{D^2}$ of the two dimensional disk with non-zero Lyapunov exponents which is identity on the boundary of the disk. We then construct a specific diffeomorphism from the interior of the disk onto an open simply connected and dense subset of $M$, which extends to a homeomorphism from the closed disk onto $M$ and is area preserving. This diffeomorphism moves $f_{D^2}$ to an area preserving Bernoulli diffeomorphism $f_M$ of the surface with non-zero Lyapunov exponents, see Section 9.

Our next step is to use the conjugacy map and a representation of $f_{T^2}$ as a Young diffeomorphism to obtain a similar representation for $f_M$, see Section 10. 
Now to obtain an upper polynomial bound for decay of correlation we use the results in \cite{Youngdecay} and \cite{mel} which allow us to choose a class of observables which includes all H\"older continuous functions on the surface. To obtain a lower bound we use the result in \cite{Gou} and the class of H\"older continuous observables on the surface which vanish inside small neighborhoods of the fixed points, see Section 10.  

Representing the map $f_M$ as a Young tower also allows us to establish the Central Limit Theorem using the results in \cite{Gou, liv} as well as the Polynomial Large Deviation property using the results in \cite{melnic}.

The paper is organized as follows. After we provide some definitions in the next section we state our main result in Section 3. In Section 4 we construct the map on the $2$-torus and state some of its properties including its class of smoothness. In Section 5, we recall the definition of Young diffeomorphisms and describe a representation of $f_{T^2}$ as a Young diffeomorphism. The proof of the main result, Theorem \ref{mainthm1}, occupies Sections 6 through 10. In Section 6, we prove some technical results that establish new crucial properties of the slow down map. In Sections 7 and 8 we obtain polynomial respectively lower and upper bounds on the tail of the return time for the Katok map $f_{T^2}$. In Section 9 we show how to carry over the Katok map of the torus a diffeomorphism of a given surface. Finally, in section 10 we complete the proof the main result.

\textbf{Acknowledments.} The first and the second authors would like to thank Bernoulli Center (CIB, Lausanne, Switzerland) where part of the work was done for their hospitality. S.S. was supported by the CNP grant.

\section{Definitions and notations}

Let $X$ be a measurable space and $T:X\to X$ a measurable invertible transformation preserving a measure $\mu$. For reader's convenience we recall some properties of the map which are of interest to us in the paper.

\subsection{The Bernoulli property} We say that $(T,\mu)$ has the \emph{Bernoulli property} if it is metrically isomorphic to the Bernoulli shift $(\sigma,\kappa)$ associated to some Lebesgue space $(Y,\nu)$, so that $\nu$ is metrically isomorphic to the Lebesgue measure on an interval together with at most countably many atoms and $\kappa$ is given as the direct product of $\mathbb{Z}$ copies of $\nu$ on $Y^{\mathbb{Z}}$. 

\subsection{Decay of correlations} Let $\mathcal{H}_1$ and $\mathcal{H}_2$ be two classes of real-valued functions on $X$ called \emph{observables}. For $h_1\in\mathcal{H}_1$ and $h_2\in\mathcal{H}_2$ define the correlation function 
$$
\text{Cor}_n(h_1,h_2):=\int h_1(T^n(x))h_2(x)\,d\mu -\int h_1(x)\,d\mu
\int h_2(x)\,d\mu.
$$
We say that $T$ has \emph{polynomial decay of correlations} (more precisely, \emph{polynomial lower bound on correlations}) with respect to classes 
$\mathcal{H}_1$ and $\mathcal{H}_2$ if there exists $\gamma_1>0$ such that for any $h_1\in\mathcal{H}_1$, $h_2\in\mathcal{H}_2$, and any $n>0$,
$$
|\text{Cor}_n(h_1, h_2)|\le Cn^{-\gamma_1},
$$
where $C=C(h_1,h_2)>0$ is a constant.

We say that $T$ admits a \emph{polynomial lower bound on correlations} with respect to classes $\mathcal{H}_1$ and $\mathcal{H}_2$ of observables if there exists $\gamma_2>0$ such that for any $h_1\in\mathcal{H}_1$,
$h_2\in\mathcal{H}_2$, and any $n>0$,
$$
|\text{Cor}_n(h_1, h_2)|\ge C'n^{-\gamma_2},
$$ 
where $C'=C'(h_1,h_2)>0$ is a constant.

We say that $T$ has \emph{exponential decay of correlations} with respect to classes $\mathcal{H}_1$ and $\mathcal{H}_2$ if there exists $\gamma_3>0$ such that for any $h_1\in\mathcal{H}_1$, $h_2\in\mathcal{H}_2$, and any $n>0$,
$$
|\text{Cor}_n(h_1, h_2)|\le C''e^{-\gamma_3n},
$$
where $C''=C''(h_1,h_2)>0$ is a constant.

\subsection{The Central Limit Theorem} We say that $T$ satisfies the \emph{Central Limit Theorem (CLT)} with respect to a class $\mathcal{H}$ of observables on $X$ if there exists $\sigma>0$ such that for any $h\in\mathcal{H}$ with $\int h=0$ the sum
$$
\frac 1{\sqrt{n}}\sum\limits_{i=0}^{n-1}h(f^i(x))
$$ 
converges in law to a normal distribution $\textit{N}(0,\sigma)$.

\subsection{The Large Deviation property} We say that $T$ has the \emph{Polynomial Large Deviation property} with respect to a class $\mathcal{H}$ of observables on $X$ if for any $\delta>0$, $\varepsilon>0$, and any $h\in\mathcal{H}$ there exists $C=C(\delta,\varepsilon, h)>0$ such that for all $n$
$$
\mu\Bigl(\Bigl|\frac{1}{n}\sum\limits_{i=0}^{n-1}h(T^i(x))-\int h\Bigr|>\varepsilon\Bigr)<Cn^{-\beta},
$$
where $\beta>0$ is a constant independent of $\delta$, $\varepsilon$, and $h$.

\subsection{Lyapunov exponents} Let $f\colon M\to M$ be a diffeomorphism of a compact smooth Riemannian manifold $M$. Given a point $x\in M$ and a vector $v\in T_xM$, the number 
$$
\chi(x,v):=\limsup_{n\to\infty}\frac1n\log\|df^n_xv\|
$$ 
is called the \emph{Lyapunov exponents} of $v$ at $x$. One can show that for every $x\in M$ the function $\chi(x,\cdot)$ takes on finitely many values which we denote by $\chi_1(x)\le\dots\le\chi_p(x)$, where $p=\dim M$. The functions 
$\chi_i(x)$, $i=1,\dots, p$ are Borel measurable and $f$-invariant.

If $\mu$ is an $f$-invariant measure, then for $\mu$-almost every $x\in M$ and any $v\in T_xM$,  
$$
\chi(x,v)=\lim_{n\to\infty}\frac1n\log\|df^n_xv\|.
$$
We say that $f$ has \emph{nonzero Lyapunov exponents} with respect to $\mu$ or that $\mu$ is \emph{hyperbolic} if for $\mu$-almost every $x$ we have that 
$\chi_i(x)\ne 0$, $i=1,\dots, p$ and that $\chi_1(x)<0$ while $\chi_p(x)>0$. 

Note that if $\mu$ is ergodic, then $\chi_i(x)=\chi_i(\mu)$ for all $i=1,\dots, p$ and $\mu$-almost every $x\in M$. 

\section{Main Results}

Let $M$ be a smooth compact connected oriented surface with area $m$. Without loss of generality we assume that $m(M)=1$. Given $\rho>0$, let $C^\rho:=C^\rho(M)$ be the class of all H\"older continuous functions on $M$.

Consider a nested sequence of subsets $\{M_j\}$ that exhaust $M$ that is 
$M_1\subset M_2\subset\ldots\subset M$ and $\bigcup_{j\ge 1}M_j=M$. Given such a sequence, let $\mathcal{G}=\mathcal{G}(\{M_j\})$ be the class of observables  
$h\subset C^\rho$ for which there is $k=k(h)$ such that $\text{supp}(h)\subset M_k$.

Given $0<\alpha<1$ and $0<\mu<1$, denote by
\begin{equation}\label{gamma-gamma'}
\gamma=\frac1{2\a}+2^{\a-1}(1+\mu)+\frac{1-\mu}6,\quad \gamma'=\frac1{2\a}+\frac{1-\mu}{2^{\a+2}}.
\end{equation}
If $0<\alpha<\frac14$ and $0<\mu<\frac12$, then clearly $\gamma>\gamma'>2$. We now state our main result. 
\begin{thm}\label{mainthm1}
Let $M$ be a compact smooth connected and oriented surface. For any 
$\frac19<\alpha<\frac14$ and $0<\mu<\frac12$ there are 
$\beta=\beta(\alpha,\mu)>0$ and an area preserving $C^{1+\beta}$ diffeomorphism $f$ of $M$ satisfying:
\begin{enumerate}
\item $f$ has the Bernoulli property;
\item $f$ has non-zero Lyapunov exponents almost everywhere with respect to $m$;
\item $f$ admits a polynomial upper bound on correlations with respect to the class $C^\rho$ of observables and a polynomial lower bound on correlations with respect to some sequence of subsets $\{M_j\}$ and the corresponding class 
$\mathcal{G}$ of observables; more precisely: 
\begin{enumerate}
\item for any $h_i\in C^\rho$, $i=1,2$ for which $\int h_1\,dm\int h_2\,dm\ne 0$,  
$$
|\text{Cor}_n(h_1, h_2)|\le C'n^{-(\gamma'-2)},
$$
where $C'=C'(\|h_1\|_{C^\rho}, \|h_2\|_{C^\rho})>0$;
\item for any $h_i\in\mathcal{G}$, $i=1,2$ for which $\int h_1dm\int h_2dm>0$,
$$
Cn^{-(\gamma-2)}\le|\text{Cor}_n(h_1, h_2)|,
$$
where $C=C(\|h_1\|_{C^\rho}, \|h_2\|_{C^\rho})>0$;
\item if $h_i\in C^\rho$, $i=1,2$ and $\int h_1\,dm= 0$, then
$$
|\text{Cor}_n(h_1, h_2)|= \mathcal{O}(n^{-(\gamma-2)});
$$
\end{enumerate}
\item $f$ satisfies the CLT for the class of observables $h\in C^\rho$, 
$\int h\,dm=0$ with $\sigma=\sigma(h)$ given by 
$$
\sigma^2=-\int h^2dm+2\sum\limits_{n=0}^{\infty}\int h\cdot h\circ f^ndm,
$$
where $\sigma>0$ if and only if $h$ is not cohomologous to zero, i.e. $h\circ f\ne g\circ f-g$ for any $g\in C^\rho$;
\item $f$ has the Polynomial Large Deviation property with respect to the class 
$C^\rho$ of observables with the constant $C$ of the form
$C=C(\|h\|_{C^\rho})\varepsilon^{-2\beta}$ where 
$\beta=\gamma'-2-\delta$ for some sufficiently small $\delta>0$. In addition, for an open and dense subset of observables in $C^\rho$ and sufficiently small $\varepsilon>0$
$$
n^{-\beta}<m\Bigl(\Bigl|\frac{1}{n}\sum\limits_{i=0}^{n-1}h(f^i(x))-\int h\Bigr|>\varepsilon\Bigr)
$$
for infinitely many $n$;
\item $f$ has a unique measure of maximal entropy (MME) with respect to which it has the Bernoulli property, non-zero Lyapunov exponents almost everywhere, exponential decay of correlations and satisfies the CLT with respect to the class $C^\rho$ of observables.
\end{enumerate}
\end{thm}
		
\section{A slow down map of the $2$-torus}

\subsection{The definition of a slow down map} Consider the automorphism of the two-dimensional torus $\T^2=\R^2/\Z^2$ given by the matrix
$A:=\left(\begin{smallmatrix} 5 & 8\\8 & 13\end{smallmatrix}\right)$. It has four fixed points $x_1=(0,0)$, $x_2=(\frac12, 0)$, $x_3=(0, \frac12)$, and 
$x_4=(\frac 12, \frac 12)$. For $i=1,2,3,4$ consider the disk 
$D_r^i=\{(s_1,s_2): {s_1}^2+{s_2}^2\le r^2\}$ of radius $r$ centered at $x_i$ and set $D_r=\bigcup_{i=1}^4D^i_r$. Here $(s_1,s_2)$ is the coordinate system obtained from the eigendirections of $A$ and originated at $x_i$. Let $\lambda>1$ be the largest eigenvalue of $A$. There are $r_1>r_0$ such that
\begin{equation}\label{eq:r1}
D_{r_0}^i\subset\text{int}A(D_{r_1}^i)\cap\text{int}A^{-1}(D_{r_1}^i)
\end{equation}
and the disks $D_{r_1}^i$ are pairwise disjoint. Fix $i$ and consider the system of differential equations in $D_{r_1}^i$
\begin{equation}\label{batata10}
\frac{ds_1}{dt}= s_1\log\lambda,\quad \frac{ds_2}{dt}=-s_2\log\lambda.
\end{equation}
Observe that $A|_{D^i_{r_0}}$ is the time-$1$ map of the local flow generated by this system.

We choose a number $0<\alpha<1$ and a function $\psi:[0,1]\mapsto[0,1]$ satisfying:
\begin{enumerate}
\item[(K1)] $\psi$ is of class $C^{\infty}$ everywhere but at the origin;
\item[(K2)] $\psi(u)=1$ for $u\ge r_0$ and some $0<r_0<1$;
\item[(K3)] $\psi'(u)> 0$ for  $0<u<r_0$;
\item[(K4)] $\psi(u)=(u/r_0)^\alpha$ for $0\le u\le\frac{r_0}{2}$.
\end{enumerate}
Using the function $\psi$, we slow down trajectories of the flow by perturbing the system \eqref{batata10} in $D_{r_0}^i$ as follows
\begin{equation}\label{batata2}
\begin{aligned}
\frac{ds_1}{dt}=&\quad s_1\psi({s_1}^2+{s_2}^2)\log\lambda\\
\frac{ds_2}{dt}=&- s_2\psi({s_1}^2+{s_2}^2)\log\lambda.
\end{aligned}
\end{equation}
This system of differential equations generates a local flow. Denote by $g^i$ the time-$1$ map of this flow. The choices of $\psi$, $r_0$ and $r_1$ (see \eqref{eq:r1}) guarantee that the domain of $g^i$ contains $D_{r_0}^i$. Furthermore, $g^i$ is of class $C^\infty$ in $D_{r_0}^i\setminus \{x_i\}$ and it coincides with $A$ in some neighborhood of the boundary $\partial D_{r_0}^i$. Therefore, the map
\begin{equation}\label{mapG}
G(x)=\begin{cases} A(x) & \text{if $x\in\T^2\setminus D_{r_0}$,}\\
g^i(x) & \text{if $x\in D_{r_0}^i$}
\end{cases}
\end{equation}
defines a homeomorphism of the torus $\T^2$, which is a $C^\infty$ diffeomorphism everywhere except at the fixed points $x_i$. Since $0<\alpha<1$, we have that
$$
\int_0^1\frac{du}{\psi(u)}<\infty.
$$
This implies that the map $G$ preserves the probability measure
\begin{equation}\label{prob-measure}
d\nu=q_0^{-1}q\,dm,
\end{equation} 
where $m$ is the area and the density $q$ is a positive $C^\infty$ function that is infinite at~$x_i$ and is defined by
$$
q(s_1,s_2):=
\begin{cases} 
(\psi({s_1}^2+{s_2}^2))^{-1}&\text{if } (s_1,s_2)\in D_{r_0}^i,\\ 
1 & \text{in } \T^2\setminus D_{r_0}
\end{cases}
$$
and
\[
q_0:=\int_{\T^2}q\,dm.
\]
We further perturb the map $G$ by a coordinate change $\phi$ in $\T^2$ to obtain an area preserving map. To achieve this, define a map $\phi$ in $D_{r_0}^i$ by the formula
\begin{equation}\label{mapshi}
\phi(s_1,s_2):=\frac{1}{\sqrt{q_0({s_1}^2+{s_2}^2)}}
\bigg(\int_0^{{s_1}^2+{s_2}^2}\frac{du}{\psi(u)}\bigg)^{1/2}(s_1,s_2)
\end{equation}
and set $\phi=\text{Id}$ in $\T^2\setminus D_{r_0}$. Clearly, $\phi$ is a homeomorphism and is a $C^\infty$ diffeomorphism outside the points $x_1,x_2,x_3,x_4$. One can show that $\phi$ transfers the measure $\nu$ into the area and that the map $f_{\T^2}=\phi\,\circ\, G\,\circ\,\phi^{-1}$ is a homeomorphism and is a $C^\infty$ diffeomorphism outside the points $x_1,x_2,x_3,x_4$. It is called a \emph{slow down map} (see \cite{katan} and also \cite{BP13}). 

The following proposition describes some basic properties of this map.

\begin{prop}[\cite{katan},\cite{BP13}]\label{smoothH}
The map $f_{\T^2}$ has the following properties:
\begin{enumerate}
\item It is topologically conjugated to $A$ via a homeomorphism $H$.
\item It admits two transverse invariant continuous stable and unstable distributions $E^s(x)$ and $E^u(x)$ and for almost every point $x$ with respect to area $m$ it has two non-zero Lyapunov exponents, positive in the direction of $E^u(x)$ and negative in the direction of $E^s(x)$. Moreover, the only invariant measure with zero Lyapunov exponents is the atomic measure supported on the fixed points $x_i$.
\item It admits two continuous, uniformly transverse, invariant foliations with smooth leaves which are the images under the conjugacy map of the stable and unstable foliations for $A$ respectively.
\item For every $\varepsilon>0$ one can choose $r_0>0$ such that 
$$
|\int\log |D\,f_{\T^2}|E^u|\,dm-\log\lambda|<\varepsilon.
$$
\item It is ergodic with respect to the area $m$.
\end{enumerate}
\end{prop}
 The following proposition establishes regularity of the map $f_{\T^2}$.
\begin{prop}\label{f is1plusbeta}
The map $f_{\T^2}$ is of class of smoothness $C^{2+2\kappa}$, where 
$\kappa=\frac{\a}{1-\a}$.
\end{prop}
\begin{proof} Fix $i\in\{1,2,3,4\}$ and consider the vector field in $D^i_{r_0}$ given by the right-hand side of \eqref{batata10}. It is Hamiltonian with respect to the area and the Hamiltonian function $H_1(s_1,s_2)=s_1s_2\log\lambda$. The vector field given by \eqref{batata2} is obtained from \eqref{batata10} by a time change and hence, is also Hamiltonian with respect to the measure $\nu$ (see \eqref{prob-measure}) and the same Hamiltonian function. The map $f_{\T^2}$ is conjugate via $\varphi$ (see \eqref{mapshi}) to the time-$1$ map of the flow generated by \eqref{batata2}. Since $\phi_*\nu=m$, $f_{\T^2}$ is the time-$1$ map of the flow which is Hamiltonian with respect to the area and the Hamiltonian function $H_2=H_1\circ\phi^{-1}$. Using \eqref{mapshi}, we find that $H_2$ can be given in $D^i_{r_0}$ as follows (see \cite{katan}):
$$
H_2(s_1,s_2)=\frac{s_1s_2 h(\sqrt{s_1^2+s_2^2})}{s_1^2+s_2^2}\log\lambda,
$$
where by $(K4)$, $h(u)=u^{\frac{2}{1-\a}}$ and $u=s_1^2+s_2^2$. To prove that $f_{\T^2}$ is of the desired class of smoothness we will show that the Hamiltonian $H_2$ has H\"older continuous partial derivatives of second order with H\"older exponent $2\kappa$. To this end we consider the function $g(x,y)=xy(x^2+y^2)^{\kappa}$ with $\kappa=\frac{\a}{1-\a}$ and show that $g$ has H\"older continuous partial derivatives of second order with H\"older exponent $2\kappa$. Note that $g$ is of class $C^{\infty}$ except for $(x,y)=(0,0)$, so we only need to show H\"older continuity of partial derivatives at the origin. Note also that the function $g$ is symmetric, so we only show that 
$\frac{\partial^2 g}{\partial x^2}$ and $\frac{\partial^2 g}{\partial x\partial y}$ are H\"older continuous. Since $\frac{\partial g}{\partial x}(0,0)=\lim\limits_{\Delta x\to 0}\frac{g(\Delta x,0)-g(0,0)}{\Delta x}=0$, we have that 
$$
\frac{\partial g}{\partial x}=
\begin{cases}
y(x^2+y^2)^{\kappa}+2\kappa x^2y(x^2+y^2)^{\kappa-1},& (x,y)\neq (0,0);\\
0& (x,y)=(0,0).\\
\end{cases}
$$
Note that $$
\frac{\partial^2 g}{\partial x\partial y}(0,0)=\lim\limits_{\Delta y\to 0}\frac{\frac{\partial g}{\partial x}(0,\Delta y)-\frac{\partial g}{\partial x}(0,0)}{\Delta y}=\lim\limits_{\Delta y\to 0}\frac{\Delta y(\Delta y)^{2\kappa}-0}{\Delta y}=0
$$
and hence,
$$
\frac{\partial^2 g}{\partial x\partial y}=\begin{cases}
(1+2\kappa)(x^2+y^2)^{\kappa}+\\
\quad +4(\kappa-1)x^2y^2(x^2+y^2)^{\kappa-2}&(x,y)\neq (0,0);\\
0 &(x,y)=(0,0).
\end{cases}
$$
Since the function $\frac{\partial^2 g}{\partial x\partial y}$ is differentiable for all 
$(x,y)\neq (0,0)$, it is H\"older continuous for all pairs of nonzero points $(x,y)$. It remains to show H\"older continuity for pairs of points one of which is zero. It is easy to see that 
$$
\big|\frac{\partial^2 g}{\partial x\partial y}(x,y)-\frac{\partial^2 g}{\partial x\partial y}(0,0)\big|
\le K(x^2+y^2)^{\kappa}=K d((x,y), (0,0))^{2\kappa},
$$
where $K>0$ and $d$ denotes the usual distance. Thus 
$\frac{\partial^2 g}{\partial x\partial y}$ is H\"older continuous with H\"older exponent $2\kappa$.

Now we consider $\frac{\partial^2 g}{\partial x^2}$. Observe that 
$\frac{\partial^2 g}{\partial x^2}(0,0)=0$ and 
$$
\frac{\partial^2 g}{\partial x^2}=
\begin{cases}
6\kappa xy(x^2+y^2)^{\kappa-1}+4(\kappa-1)x^3y(x^2+y^2)^{\kappa-2} & (x,y)\neq (0,0);\\ 
0, & (x,y)=(0,0). 
\end{cases} 
$$ 
It is easy to see that 
$$
\big|\frac{\partial^2 g}{\partial x^2}(x,y)-\frac{\partial^2 g}{\partial x^2}(0,0)\big|\le
K(x^2+y^2)^{\kappa}=K d((x,y), (0,0))^{2\kappa}.
$$
Hence, $\frac{\partial^2 g}{\partial x^2}$ is H\"older continuous with H\"older exponent $2\kappa$.
\end{proof}

\subsection{One dimensional maps with indifferent fixed point}
Various authors including \cite{Gou, Hu, LSV} have studied intermittency (Pomeau-Maneville)maps of the type $T:[0,1]\to [0,1]$ given by
$$
T(x)=\begin{cases}
x(1+2^{\a}x^{\a}) &0\le x\le \frac 12\\
2x-1 & \frac 12< x\le 1
\end{cases}
$$
where $0<\a<1$, for which zero is an indifferent fixed point. It is shown in \cite{Hu} that $T$ has polynomial decay of correlations with exponent 
$\gamma=\frac1{\a}-1$ for the class of Lipschitz continuous functions and there are Lipschitz continuous functions $h_1$ and $h_2$ for which $T$ has a polynomial lower bound with the same exponent $\gamma$.

We would like to emphasize that the methods for obtaining the above correlations estimates are substantially different from ours. Indeed, they use the representation of the map $T$ as a renewal shift (see for example, \cite{Hu}). The latter is a particular case of Young tower (see the next  section) with the inducing time $\tau$ to be the first return time to the base of the tower $\Lambda=[\frac12,1]$. The crucial feature of this tower is that there exists exactly one partition element $\Lambda_n\subset [\frac12,1]$ with the return time 
$\tau(\Lambda_n)=n$. While for the slow-down map we also construct a Young tower representation for which the induced time is the first return time to the base, the number of partition elements with a given inducing time may (and do) grow exponentially. 

\section{Proof of Theorem \ref{mainthm1}: representing $f_{\T^2}$ as a Young diffeomorphism} 

\subsection{Young diffeomorphisms}

Consider a $C^{1+\epsilon}$ diffeomorphism $f:M\to M$ of a compact smooth Riemannian manifold $M$. Following \cite{Youngannals} we describe a collection of conditions on the map $f$. \footnote{Our requirements are slightly different than those in \cite{Youngannals} since we do not assume the inducing domain $\Lambda$ to be compact.}

An embedded $C^1$-disk $\gamma\subset M$ is called an \emph{unstable disk} (respectively, a \emph{stable disk}) if for all $x,y\in\gamma$ we have that
$d(f^{-n}(x), f^{-n}(y))\to 0$ (respectively, $d(f^{n}(x), f^{n}(y))\to 0$) as $n\to+\infty$. A collection of embedded $C^1$ disks $\Gamma^u=\{\gamma^u\}$\label{sb:gamma} is called a \emph{continuous family of unstable disks} if there exists a homeomorphism
$\Phi: K^s\times D^u\to\cup\gamma^u$ satisfying:
\begin{itemize}
\item $K^s\subset M$ is a Borel subset and $D^u\subset \R^d$ is the closed unit disk for some $d<\dim M$;
\item $x\to\Phi|{\{x\}\times D^u}$ is a continuous map from $K^s$ to the space of $C^1$ embeddings of $D^u$ into $M$ which can be extended to a continuous map of the closure 
$\overline{K^s}$;
\item $\gamma^u=\Phi(\{x\}\times D^u)$ is an unstable disk.
\end{itemize}
A \emph{continuous family of stable disks} is defined similarly.

We allow the sets $K^s$ to be non-compact in order to deal with overlaps which appear in most known examples including the Katok map.

A set $\Lambda\subset M$ has \emph{hyperbolic product structure} if there exists a continuous family $\Gamma^u=\{\gamma^u\}$ of unstable disks $\gamma^u$ and a continuous family $\Gamma^s=\{\gamma^s\}$ of stable disks $\gamma^s$ such that
\begin{itemize}
\item $\text{dim }\gamma^s+\text{dim }\gamma^u=\text{dim } M$;
\item the $\gamma^u$-disks are transversal to $\gamma^s$-disks with an angle uniformly bounded away from $0$;
\item each $\gamma^u$-disks intersects each $\gamma^s$-disk at exactly one point;
\item $\Lambda=(\cup\gamma^u)\cap(\cup\gamma^s)$.
\end{itemize}

A subset $\Lambda_0\subset\Lambda$ is called an \emph{$s$-subset} if it has hyperbolic product structure and is defined by the same family $\Gamma^u$ of unstable disks as 
$\Lambda$ and a continuous subfamily $\Gamma_0^s\subset\Gamma^s$ of stable disks. A \emph{$u$-subset} is defined analogously.

We define the \emph{$s$-closure} $scl(\Lambda_0)$ of an \emph{$s$-subset} 
$\Lambda_0\subset\Lambda$ by
$$  
scl(\Lambda_0):= \bigcup_{x\in \overline{\Lambda_0\cap\gamma^u}} \gamma^s(x)\cap \Lambda 
$$
and the \emph{u-closure} $ucl(\Lambda_1)$ of a given \emph{u-subset} 
$\Lambda_1\subset \Lambda$ similarly:
$$ 
ucl(\Lambda_1):= \bigcup_{x\in \overline{\Lambda_1\cap \gamma^s}} \gamma^u(x)\cap \Lambda.  
$$ 
Assume the map $f$ satisfies the following conditions:
\begin{enumerate}
\item[(Y1)] There exists $\Lambda\subset M$ with hyperbolic product structure, a countable collection of continuous subfamilies $\Gamma_i^s\subset\Gamma^s$ of stable disks and positive integers $\tau_i$, $i\in\mathbb{N}$ such that the $s$-subsets 
\begin{equation}
\Lambda_i^s:=\bigcup_{\gamma\in\Gamma^s_i}\,\bigl(\gamma\cap \Lambda\bigr)\subset\Lambda
\end{equation}
are pairwise disjoint and satisfy:
\begin{enumerate}
\item \emph{invariance}: for every $x\in\Lambda_i^s$ 
$$
f^{\tau_i}(\gamma^s(x))\subset\gamma^{s}(f^{\tau_i}(x)), \,\, f^{\tau_i}(\gamma^u(x))\supset\gamma^u(f^{\tau_i}(x)),
$$
where $\gamma^{u,s}(x)$ denotes the (un)stable disk containing $x$;
\item \emph{Markov property}: $\Lambda_i^u:=f^{\tau_i}(\Lambda_i^s)$ is a $u$-subset of 
$\Lambda$ such that for all $x\in\Lambda_i^s$ 
$$
\begin{aligned}
f^{-\tau_i}(\gamma^s(f^{\tau_i}(x))\cap\Lambda_i^u)
&=\gamma^s(x)\cap \Lambda,\\
f^{\tau_i}(\gamma^u(x)\cap\Lambda_i^s)
&=\gamma^u(f^{\tau_i}(x))\cap \Lambda.
\end{aligned}
$$
\end{enumerate}
\end{enumerate}
\vspace*{.5cm}
\begin{enumerate}
\item[(Y2)] The sets $\Lambda_i^u$ are pairwise disjoint.
\end{enumerate}
For any $x\in \Lambda^s_i$ define the \emph{inducing time} by $\tau(x):=\tau_i$ and the \emph{induced map} $\tilde{f}: \bigcup_{i\in\mathbb{N}}\Lambda_i^s\to\Lambda$ by
$$
\tilde{f}|_{\Lambda_i^s}:=f^{\tau_i}|_{\Lambda_i^s}.
$$
\begin{enumerate}
\item[(Y3)] There exists $0<a<1$ such that for any $i\in\mathbb{N}$ we have:
\begin{enumerate}
\item[(a)] For $x\in\Lambda_i^s$ and $y\in\gamma^s(x)$,
$$
d(\tilde{f}(x), \tilde{f}(y))\le a\, d(x,y);
$$
\item[(b)] For $x\in\Lambda^s_i$ and $y\in\gamma^u(x)\cap \Lambda_i^s$,
$$
d(x,y)\le a \, d(\tilde{f}(x), \tilde{f}(y)).
$$
\end{enumerate}
\end{enumerate}
For $x\in \Lambda$ let $Jac f(x)=\det |Df|_{E^u(x)}|$ and  
$Jac \tilde{f}(x)=\det |D\tilde{f}|_{E^u(x)}|$ denote the Jacobian of $Df|_{E^u(x)}$ and $D\tilde{f}|_{E^u(x)}$ respectively. 
\begin{enumerate}
\item[(Y4)] There exist $c>0$ and $0<\kappa<1$ such that:
\begin{enumerate}
\item[(a)] For all $n\ge 0$, $x\in \tilde{f}^{-n}(\cup_{i\in\mathbb{N}}\Lambda^s_i)$ 
and $y\in\gamma^s(x)$ we have
\[
\left|\log\frac{Jac \tilde{f}(\tilde{f}^{n}(x))}{Jac \tilde{f}(\tilde{f}^{n}(y))}\right|\le c\kappa^n;
\]
\item[(b)] For any $i_0,\dots, i_n\in\mathbb{N}$, 
$\tilde{f}^k(x),\tilde{f}^k(y)\in\Lambda^s_{i_k}$ for $0\le k\le n$ and  
$y\in\gamma^u(x)$ we have
\[
\left|\log\frac{ Jac \tilde{f}(\tilde{f}^{n-k}(x))}{Jac \tilde{f}(\tilde{f}^{n-k}(y))}\right|\le c\kappa^k.
\]
\end{enumerate}
\end{enumerate}

\begin{enumerate}
\item[(Y5)] For every $\gamma^u\in\Gamma^u$ one has
$$
\mu_{\gamma^u}(\gamma^u\cap \Lambda)>0, \quad \mu_{\gamma^u}\left((\overline{\Lambda\setminus\cup\Lambda_i^s)\cap\gamma^u}\right)=0,
$$
where $\mu_{\gamma^u}$ is the leaf volume on $\gamma^u$.
\end{enumerate}

\begin{enumerate}
\item[(Y6)] There exists $\gamma^u\in\Gamma^u$ such that 
$$
\sum_{i=1}^\infty \tau_i \mu_{\gamma^u}(\Lambda_i^s\cap \gamma^u) <\infty.
$$
\end{enumerate}
It is shown in \cite{THK} that the Katok map is a Young diffeomorphism. We will briefly outline the argument.

\subsection{A tower representation of the automorphism $A$} Consider a finite Markov partition $\tilde{\Px}$ for the automorphism $A$. Recall that by definition of Markov partitions, $\tilde{P}=\overline{\text{int}\tilde{P}}$ for any 
$\tilde{P}\in\tilde{\Px}$. Let $\tilde{P}\in\tilde{\Px}$ be a partition element which does not intersect any of the disks $D^i_{r_0}$, $i=1,2,3,4$. Given $\delta>0$, we can always choose the Markov partition $\tilde{\Px}$ in such a way that 
$\text{diam }(\tilde{P})<\delta$. For a point $x\in\tilde{P}$ denote by 
$\tilde{\gamma}^s(x)$ (respectively, $\tilde{\gamma}^u(x)$) the connected component of the intersection of $\tilde{P}$ with the stable (respectively, unstable) leaf of $x$, which contains $x$. We say that $\tilde{\gamma}^s(x)$ and $\tilde{\gamma}^u(x)$ are \emph{full length} stable and unstable curves through $x$.

Given $x\in\tilde{P}$, let $\tilde{\tau}(x)$ be the first return time of $x$ to 
$\text{int}\tilde{P}$. For all $x$ with $\tilde{\tau}(x)<\infty$ denote by
$$
\tilde{\Lambda}^s(x)=\bigcup_{y\in\tilde{U}^u(x)\setminus\tilde{A}^u(x)}\,\tilde{\gamma}^s(y),
$$ 
where $\tilde{U}^u(x)\subseteq\tilde{\gamma}^u(x)$ is an interval containing 
$x$ and open in the induced topology of $\tilde{\gamma}^u(x)$, and 
$\tilde{A}^u(x)\subset\tilde{U}^u(x)$ is the set of points which either lie on the boundary of the Markov partition or never return to the set $\tilde{P}$. Note that $\tilde{A}^u(x)$ has zero one-dimensional Lebesgue measure in $\tilde{\gamma}^u(x)$. One can choose 
$\tilde{U}^u(x)$ such that
\begin{enumerate}
\item for any $y\in\tilde{\Lambda}^s(x)$ we have $\tilde{\tau}(y)=\tilde{\tau}(x)$;
\item for any $y\in\tilde{P}$ such that $\tilde{\tau}(y)=\tilde{\tau}(x)$ we have 
$y\in\tilde{\Lambda}^s(x)$.
\end{enumerate}
Moreover, the image under $A^{\tilde{\tau}(x)}$ of $\tilde{\Lambda}^s(x)$ is a $u$-subset containing $A^{\tilde{\tau}(x)}(x)$. It is easy to see that for any $x,y\in\tilde{P}$ with finite first return time the sets $\tilde{\Lambda}^s(x)$ and $\tilde{\Lambda}^s(y)$ are either coincide or disjoint. Thus we have a countable collection of disjoint sets 
$\tilde{\Lambda}_i^s$ and numbers $\tilde{\tau}_i$ which give a representation of the automorphism $A$ as a Young diffeomorphism for which the set 
$$
\tilde{\Lambda}=\bigcup_{i\ge 1}\tilde{\Lambda}_i^s
$$ 
is the base of the tower, the sets $\tilde{\Lambda}_i^s$ are the $s$-sets and the numbers 
$\tilde{\tau}_i$ are the inducing times, see \cite{THK} for details. 

\subsection{A tower representation for the slow down map $f_{\T^2}$}\label{torustower} Applying the conjugacy map $H$, one obtains the element $P=H(\tilde{P})$ of the Markov partition $\mathcal{P}=H(\tilde{\mathcal{P}})$. Since the map $H$ is continuous, given 
$\varepsilon$, there is $\delta>0$ such that $\text{diam }(P)<\varepsilon$ for any $P\in\mathcal{P}$ provided $\text{diam }(\tilde{P})<\delta$. Further we obtain the set 
$\Lambda=H(\tilde{\Lambda})$, which has direct product structure given by the full length stable $\gamma^s(x)=H(\tilde{\gamma}^s(x))$ and unstable 
$\gamma^u(x)=H(\tilde{\gamma}^u(x))$ curves. We thus obtain a representation of the slow down map as a Young diffeomorphism for which $\Lambda_i^s=H(\tilde{\Lambda}_i^s)$ are $s$-sets, $\Lambda_i^u=H(\tilde{\Lambda}_i^u)=f_{\T^2}^{\tau_i}(\Lambda_i^s)$ are 
$u$-sets and the inducing times $\tau_i=\tilde{\tau}_i$ are the first return time to 
$\Lambda$. Note that for all $x$ with $\tau(x)<\infty$ 
$$
\Lambda^s(x)=\bigcup_{y\in U^u(x)\setminus A^u(x)}\,\gamma^s(y),
$$ 
where $U^u(x)=H(\tilde{U}^u(x))\subseteq\gamma^u(x)$ is an interval containing $x$ and open in the induced topology of $\gamma^u(x)$, and 
$A^u(x)=H(\tilde{A}^u(x))\subset U^u(x)$ is the set of points which either lie on the boundary of the Markov partition or never return to the set $P$. Note that $A^u(x)$ has zero one-dimensional Lebesgue measure in $\gamma^u(x)$.

In what follows we will always assume that a Markov partition and the slow down domain are chosen such that the following statement holds.
\begin{prop}\label{partition}
Given $Q>0$, one can choose a Markov partition $\mathcal{P}$ and the number $r_0$ in the construction of the map $f_{\T^2}$ such that 
\begin{enumerate} 
\item there is a partition element $P$ for which $f_{\T^2}^j(x)\notin D_{r_0}$ for any $0\le j\le Q$ and for any point $x$ for which either $x\in\Lambda$ or $x\notin f_{\T^2}(D_{r_0})$ while $f_{\T^2}^{-1}(x)\in D^i_{r_0}$ for some $i=1,2,3,4$.
\item if $P_i$ is the element of the Markov partition containing $x_i$, $i=1,2,3,4$, then $x_i\in D^i_{r_0}\subset\text{Int }P_i$.
\end{enumerate}
\end{prop}
To prove this proposition observe that it holds for the automorphism $A$ and hence, it remains to apply the conjugacy homeomorphism $H$.

\begin{prop}[\cite{THK}]\label{f is Young}
There exists $Q>0$ such that the collection of $s$-subsets $H(\Lambda_i^s)$ satisfies Conditions (Y1)-(Y6).
\end{prop}

\subsection{Lifting the slow down map to the tower} We define \emph{Young tower} with the base $\Lambda$ by setting 
$$
\hat{Y}=\{(x,k)\in\Lambda\times\N: 0\le k< \tau(x)\}.
$$ 
and the \emph{tower map} $\hat{f}:\hat{Y}\to\hat{Y}$ by $\hat{f}(x,k)=(x,k+1)$ if $k<\tau(x)-1$ and $\hat{f}(x,k)=(Fx,0)$ if $k=\tau(x)-1$ where 
$F:\Lambda\to\Lambda$ is the induced map and is given by $F(x)=f^{\tau(x)}(x)$. The map $\hat{f}$ is the lift of the slow down map to the tower and it preserves the \emph{lift measure} $\hat{m}=m\times\text{counting}/(\int_{\Lambda}\tau)$. We have that $\hat{f}$ is a measurable bijection from 
$\Lambda_i^s\times\{k-1\}$ to $\Lambda_i^s\times\{k\}$ for all 
$1\le k< \tau_i-1$ and from $\Lambda_i^s\times\{\tau_i-1\}$ to $\Lambda\times\{0\}$.

\section{Proof of Theorem \ref{mainthm1}: Technical Lemmas}

We establish here several technical results on the solutions of the nonlinear systems of differential equations \eqref{batata2}. Throughout this section we fix a number $0<\alpha<1$ and $i\in\{1,2,3,4\}$ and for simplicity we drop the index $i$ in the notation of the disk $D^i_r$. We also set $f:=f_{\T^2}$.

\begin{lem}[\cite{THK}, Lemma 5.1]\label{var-eq1}
For $s=(s_1,s_2)\in D_{\frac{r_0}{2}}$ let 
$$
d_{i,j}=d_{i,j}(s_1,s_2):=\frac{\partial^2 }{\partial s_i\partial s_j}s_2\psi(s_1^2+s_2^2).
$$ 
Then 
\[
\max_{i,j=1,2}\ \left|d_{i,j}\right|\le \frac{6\alpha}{r_0^\alpha}\, (s_1^2+s_2^2)^{\alpha-\frac12}.
\]
\end{lem}
Consider a solution $s(t)=(s_1(t),s_2(t))$ of  Equation~\eqref{batata2} with an initial condition $s(0)=(s_1(0),s_2(0))$. Assume it is defined on the {\it maximal} time interval $[0,T]$ for which $f^{-1}(s(0))\notin D_{\frac{r_0}{2}}$ and $f(s(T))\notin D_{\frac{r_0}{2}}$ but $s(t)\in D_{\frac{r_0}{2}}$ for all $0\le t\le T$. In particular, $s_1(t)\neq 0$ and 
$s_2(t)\neq 0$. Setting $T_1=\frac{T}{2}$ we have that $s_1(t)\le s_2(t)$ for all $0\le t\le T_1$ and $s_1(t)\ge s_2(t)$ for all $T_1\le t\le T$. The following statement provides effective lower and upper bounds on the functions $s_1(t)$ and $s_2(t)$. 

\begin{lem}\label{lem5.2}
The following statements hold: 
$$
\begin{aligned}
|s_2(t)|&\ge |s_2(a)|\left(1+2^\alpha\,C_1 \,s_2^{2\alpha}(a)\,(t-a)\right)^{-\frac{1}{2\alpha}},& 0\le  a\le t\le T_1;\\
|s_2(t)|&\le |s_2(a)|\left(1+C_1\,s_2^{2\alpha}(a)\,(t-a)\right)^{-\frac{1}{2\alpha}},
& 0\le  a\le t\le T;\\
|s_1(t)|&\ge |s_1(b)| \left( 1 + 2^\alpha C_1 s_1^{2\alpha}(b)(b - t) \right)^{-\frac{1}{2\alpha}}, 
&  T_1 \le t\le b \le T;\\
|s_1(t)|&\le |s_1(b)| \left( 1 + C_1 s_1^{2\alpha}(b)(b - t) \right)^{-\frac{1}{2\alpha}}, 
& 0 \le t\le  b \le T;\\
|s_1(t)|&\le |s_1(T_1)|\left(1-2^\alpha\,C_1\, s_1^{2\alpha}(T_1)\,(t-T_1)\right)^{-\frac{1}{2\alpha}},& T_1\le t\le T. 
\end{aligned}
$$
where $C_1=\frac{2\alpha \log\lambda}{r_0^\alpha}$ is a constant. 
\end{lem}
\begin{proof} The first four estimated are proven in \cite{THK} (see Lemma 5.2 in the erratum to the paper) and the fifth estimate can be easily shown in the same manner as in \cite{THK}. 
\end{proof}
Consider another solution $\tilde{s}(t)=(\tilde{s}_1(t),\tilde{s}_2(t))$ of Equation \eqref{batata2} satisfying an initial condition $\tilde{s}(0)=(\tilde{s}_1(0),\tilde{s}_2(0))$. For $i=1,2$, we set 
$$
\Delta s_i(t)=\tilde{s}_i(t)-s_i(t).
$$
For the proof of the next result see \cite{THK}, Lemma 5.3 and the erratum to the paper. 
\begin{lem}\label{deltas2upp}
Given $0<\mu<1$. Assume that $s_1(t)\neq 0\neq s_2(t)$ and that
\begin{enumerate}
\item $\Delta s_2(t)>0$ and $|\Delta s_1(t)|\le\mu\Delta s_2(t)$ for $t\in[0,T]$; 
\item $\Big|\frac{\Delta s_2}{s_2}(0)\Big|<\frac{1-\mu}{72}$.
\end{enumerate}
Then 
$$
\begin{aligned}
\Delta s_2(t)&\le\frac{\Delta s_2(0)}{s_2(0)}s_2(t)\left(1+2^\alpha\,C_1\, s_2^{2\alpha}(0)\,t\,\right)^{-\beta'},& 0\le t\le T_1;\\
\Delta s_2(t) & \le \frac{\Delta s_2(T_1)}{s_1(T_1)}s_1(t)
\left( \frac{1 + 2^\alpha C_1 s_1^{2\alpha}(b)(b - t)}{1 + 2^\alpha C_1 s_1^{2\alpha}(b)(b - T_1)} \right)^{\beta'}, & T_1 \le t \le b\le T,
\end{aligned}
$$
where $\beta'=\frac{1-\mu}{2^{\alpha+2}}$ and $C_1$ is the constant in Lemma~\ref{lem5.2}. In addition, 
$$
\|\Delta s(T)\|\le\sqrt{1+\mu^2}\ \frac{s_1(T)}{s_2(0)}\ \|\Delta s(0)\|. 
$$
\end{lem}

\begin{lem}\label{deltas2upp1}
Under the assumptions of Lemma \ref{deltas2upp} for any $0<t<T_1$ we have that
$$
\Delta s_2(t)< C_2\Delta s_2(0)t^{-\gamma'},
$$
where $C_2>0$ and $\gamma'=\frac1{2\a}+\frac{1-\mu}{2^{\a+2}}$ (see \eqref{gamma-gamma'}).
\end{lem}
\begin{proof} By Lemma \ref{lem5.2} (the second estimate), one has
$$
s_2(t)\le s_2(0)(1+C_1s_2^{2\a}(0)t)^{-\frac1{2\a}}.
$$
Therefore, Lemma \ref{deltas2upp} implies that
$$
\begin{aligned}
\Delta s_2(t)&\le\frac{\Delta s_2(0)}{s_2(0)}s_2(t)\left(1+2^\alpha\,C_1\, s_2^{2\alpha}(0)\,t\,\right)^{-\beta'}\\
&\le\Delta s_2(0)\left(1+C_1\, s_2^{2\alpha}(0)\,t\,\right)^{-\gamma'}.
\end{aligned}
$$
Since $(1+C_1 s_2^{2\alpha}(0))t>C_1 s_2^{2\alpha}(0)t$, we have
$$
\Delta s_2(t)\le \Delta s_2(0)(C_1s_2^{2\a}(0))^{-\gamma'}t^{-\gamma'}
$$
and the desired estimate follows, since $s_2(0)$ is of order $r_0$. 
\end{proof}

\begin{lem}\label{deltas2upp2} Under the assumptions of Lemma \ref{deltas2upp} there is 
$C_3>0$ such that for all $T_1<t<T$ we have 
$$
\Delta s_2(t)<C_3\Delta s_2(T_1).
$$
\end{lem}
\begin{proof}
Let $s(t)=(s_1(t), s_2(t))$ be the solution of \eqref{batata2} with $s_1(t)s_2(t)=\kappa_0,$ where $\kappa_0>0$ is a constant. Then we have
$$\frac{ds_2}{dt}=-\log\lambda s_2(s_1^2+s_2^2)^{\a}\le -\log\lambda s_2(2\kappa_0)^{\a}.$$
By integrating the above over $[0,T_1]$, we obtain
$$
\log\frac{s_2(T_1)}{s_2(0)}\le-\log\,\lambda(2\kappa_0)^{\a}T_1.
$$
It follows that
$$
\begin{aligned}
1-2^\alpha C_1 s_1^{2\alpha}(T_1)(t-T_1)&=1-2^\alpha C_1 s_2^{2\alpha}(T_1)(t-T_1)\\
&\ge 1-2^\alpha C_1 s_1^{2\alpha}(0)\exp(-2\a\log\lambda(2\kappa_0)^{\a}T_1)(t-T_1)\\
&\ge  1-2^\alpha C_1 s_1^{2\alpha}(0)\exp(-2\a\log\lambda(2\kappa_0)^{\a}T_1)T_1\\
&\ge 1-C_0>0
\end{aligned}
$$
for some $C_0>0$. Therefore, the second estimate in Lemma \ref{deltas2upp}, the fifth inequality in Lemma \ref{lem5.2}, and the assumption that $T_1<t<T$ imply
$$
\begin{aligned}
\frac{\Delta s_2(t)}{\Delta s_2(T_1)}&\le \left(1-2^\alpha C_1 s_1^{2\alpha}(T_1)(t-T_1)\right)^{-\frac{1}{2\a}}\left( \frac{1 + 2^\alpha C_1 s_1^{2\alpha}(b)(b - t)}{1 + 2^\alpha C_1 s_1^{2\alpha}(b)(b - T_1)} \right)^{\beta'}\\
&\le (1-C_0)^{-\frac{1}{2\a}},
\end{aligned}
$$
which proves the desired estimate.
\end{proof}
\begin{lem}\label{deltas2low1}Under the assumptions of Lemma \ref{deltas2upp} one has
\begin{align*}
\Delta s_2(t)&\ge\frac{\Delta s_2(0)}{s_2(0)}s_2(t)\left(1+C_1\, s_2^{2\alpha}(0)\,t\,\right)^{-\beta},& 0\le t\le T_1;\\
\Delta s_2(t)&\ge\frac{\Delta s_2(T_1)}{s_1(T_1)}s_1(t)\left(1+C_1\, s_1^{2\alpha}(T_1)\,(t-T_1)\right)^{-\beta_1}, & T_1 \le t \le T, 
\end{align*}
where 
$$
\beta=(1+\mu)2^{\a-1}+\frac{1-\mu}{6} \text{ and } \beta_1=2^{\alpha-1}(1+\mu)+\frac{2^\alpha}{\a}+\frac{1-\mu}{6}.
$$
\end{lem}
\begin{proof}
Let $s_1=s_1(t)$, $s_2=s_2(t)$, $u:=s_1^2+s_2^2$, and 
$\tilde{u}=\tilde{s}_1^2+\tilde{s}_2^2$. Assume $s_1(t)$ and $s_2(t)$ are strictly positive (the proof in the case when $s_1(t)$ and $s_2(t)$ are strictly negative follows by symmetry). 
By Equation \eqref{batata2}, we have
\begin{equation}\label{deltas}
\begin{aligned}
\frac{d}{dt}\Delta s_2(t)&=\frac{d}{dt}\tilde{s}_2(t)-\frac{d}{dt}s_2(t)=-(\log\lambda)
\left(\tilde{s}_2\psi(\tilde u)-s_2\psi(u)\right) \\
&=-\log\lambda\left(\frac{\partial}{\partial s_1}\Big(s_2\psi(u)\Big)\Delta s_1+\frac{\partial}{\partial s_2}\Big(s_2\psi(u)\Big)\Delta s_2\right)\\
&\hspace{0.5cm}
-\frac{\log\lambda}{2} \sum_{i,j=1,2}d_{i,j}(\xi_1,\xi_2)(\Delta s_i)(\Delta s_j)
\end{aligned}
\end{equation}
for some $\xi=(\xi_1,\xi_2)$ for which $\xi_i$ lies between $s_i(t)$ and $\tilde{s}_i(t)$ for $i=1,2$ (see Lemma \ref{var-eq1} for the definition of $d_{i,j}(\xi_1,\xi_2)$). Note that 
$$
\frac{\partial}{\partial s_1}\Big(s_2\psi(u)\Big)=2s_1s_2\psi',
\quad\frac{\partial}{\partial s_2}\Big(s_2\psi(u)\Big)=2s_2^2\psi'+\psi
$$
and hence,
$$
\begin{aligned}
\frac{d}{dt}\left(\frac{\Delta s_2}{s_2}\right)
&=\frac{1}{s_2}\left(\frac{d}{dt}\Delta s_2\right)-\frac{\Delta s_2}{s_2^2}\left(\frac{ds_2}{dt}\right)\\
&=-\log\lambda\left(2\psi'(s_1 \Delta s_1 + s_2 \Delta s_2)
+\frac{\Delta s_2}{s_2}\psi\right)\\
&\hspace{0.5cm}+(\log\lambda)\frac{\Delta s_2}{s_2}\psi
-\frac{\log\lambda}{2} \sum_{i,j=1,2}d_{i,j}(\xi_1,\xi_2)\frac{\Delta s_i\Delta s_j}{s_2}\\
&=-\frac{2\alpha\log\lambda}{r_0^\alpha}\, (s_1^2+s_2^2)^{\alpha-1}(s_1 \Delta s_1 + s_2 \Delta s_2)\\
&\hspace{0.5cm}
-\frac{\log\lambda}{2} \sum_{i,j=1,2}d_{i,j}(\xi_1,\xi_2)\frac{\Delta s_i\Delta s_j}{s_2}. 
\end{aligned}
$$
Note that for $0\le t\le T_1$ we have $0<s_1(t)\le s_2(t)$. Since $|\Delta s_1|\le\mu\Delta s_2<\Delta s_2$, we have
$$
s_1\Delta s_1+s_2\Delta s_2\le (s_1\mu+s_2)\Delta s_2\le (1+\mu)s_2\Delta s_2
$$
and Lemma~\ref{var-eq1} yields
\begin{equation}\label{eq:dij}
\sum_{i,j=1,2}d_{i,j}(\xi_1,\xi_2)\Delta s_i\Delta s_j
\le \frac{24\,\alpha}{r_0^\alpha}\ (\xi_1^2 +\xi_2^2)^{\alpha-\frac12}(\Delta s_2)^2.
\end{equation}
Therefore, we have
$$
\begin{aligned}
\frac{d}{dt}(\frac{\Delta s_2}{s_2})
&\ge -(1+\mu)\frac{2\a \log\lambda}{r_0^{\a}}(s_1^2+s_2^2)^{\a-1}s_2^2\frac{\Delta s_2}{s_2}\\
&-\frac{12\a \log\lambda}{r_0^{\a}}s_2^{2\a}(\frac{\xi_1^2 +\xi_2^2}{s_2^2})^{\alpha-\frac12}\Big(\frac{\Delta s_2}{s_2}\Big)^2.
\end{aligned}
$$
Using again the fact that $0<s_1(t)\le s_2(t), \ 0<t<T_1$ we obtain
$$
\begin{aligned}
\frac{d}{dt}\Big(\frac{\Delta s_2}{s_2}\Big)
&\ge -(1+\mu)\frac{\a \log\lambda}{r_0^{\a}}2^{\a}(s_2)^{2\a}\frac{\Delta s_2}{s_2}\\
&-\frac{12\a \log\lambda}{r_0^{\a}}s_2^{2\a}(\frac{\xi_1^2 +\xi_2^2}{s_2^2})^{\alpha-\frac12}\Big(\frac{\Delta s_2}{s_2}\Big)^2.
\end{aligned}
$$
Let $\chi=\chi(t)=\frac{\Delta s_2}{s_2}(t).$ Then the above inequality can be written as
$$
\frac{d\chi}{dt}\ge -\frac{\alpha\log\lambda}{r_0^\alpha}s_2^{2\alpha}\chi\Big((1+\mu)2^{\a} + 12\left(\frac{\xi_1^2+\xi_2^2}{s_2^2}\right)^{\alpha-\frac12}\chi\Big).
$$
Following arguments in \cite{THK} (see page 17) one can derive from here that 
$\Bigl(\frac{\xi_1^2+\xi_2^2}{s_2^2}\Bigr)^{\alpha-\frac12}\le 2$ and that
$$
\frac{d\chi}{dt}\ge -A\frac{\alpha\log\lambda}{r_0^\alpha}s_2^{2\alpha}(t)\chi (t),
$$
where $A=(1+\mu)2^{\a}+\frac{1-\mu}{3}$. By Gronwall's inequality (applied to $-\chi(t)$) and the second inequality in Lemma \ref{lem5.2}, we obtain that
\[  	
\begin{aligned}
\chi(t)&\ge\chi(0)\exp\left(-A\frac{\alpha\log\lambda}{r_0^\alpha} 
\int_0^t s_2^{2\alpha}(\tau)\,d\tau\right) \\
&\ge\chi(0)\exp\left(-A\frac{\alpha\log\lambda}{r_0^\alpha} 
\int_0^t s_2^{2\alpha}(0) (1+C_1 s_2^{2\alpha}(0)\,\tau)^{-1}\,d\tau\right) \\
&=\chi(0)\exp\left(-A\frac{\alpha\log\lambda}{r_0^\alpha} 
\frac{1}{C_1}\log(1+C_1 s_2^{2\alpha}(0)\,t\ )\right)\\ 
&=\chi(0)\exp\left(-\frac{A}{2}\log(1+C_1s_2^{2\alpha}(0)\,t\ )\right)\\
&=\chi(0)\left(1+C_1s_2^{2\alpha}(0)\,t\ \right)^{-\beta},
\end{aligned}
\]
where $\beta=\frac{A}{2}=(1+\mu)2^{\a-1}+\frac{1-\mu}{6}$.

In the case $s_2<0$ one can show using argument similar to the above that the same estimate for $\chi(t)$ holds but with exponent $(1+\mu)2^{\a-1}-\frac{1-\mu}{6}<\beta$. This completes the proof of the first estimate.

To prove the second estimate, using \eqref{deltas}, we obtain that
$$
\begin{aligned}
\frac{d}{dt}\left( \frac{\Delta s_2}{s_1}\right)
&=-\log\lambda \left( 2s_2\psi' \Delta s_1+(2s_2^2\psi' +\psi)\frac{\Delta s_2}{s_1}\right)
-\log\lambda \psi \frac{\Delta s_2}{s_1} \\
&\hspace{0.5cm}-\frac{\log \lambda}{2}\sum_{i,j=1,2}d_{i,j}(\xi_1,\xi_2)
\frac{\Delta s_i\Delta s_j}{s_1}.
\end{aligned}
$$
By the assumption $|\Delta s_1|\le \mu\Delta s_2$ and positivity of 
$s_1$, $s_2$, $\psi'$, and $\Delta s_2$, we obtain
$$
\begin{aligned}
\frac{d}{dt}\left( \frac{\Delta s_2}{s_1}\right)
&\ge-\log\lambda ( 2\mu s_1s_2\psi'+2s_2^2\psi'+2\psi )\frac{\Delta s_2}{s_1}\\
&\hspace{0.5cm}-\frac{\log \lambda}{2}\sum_{i,j=1,2}d_{i,j}(\xi_1,\xi_2)
\frac{\Delta s_i\Delta s_j}{s_1}.
\end{aligned}
$$
Since $s_2(t)\le s_1(t)$ on $[T_1,T]$, we have
$$ 
2\mu s_1s_2\psi'+2s_2^2\psi'+2\psi=2\psi'(\mu s_1^2+s_1^2+\frac 1{\a}2s_1^2)\le 2^{\a}\frac{\a\mu+\a+2}{r_0^{\a}}s_1^{2\a}.
$$
Using this fact along with the estimate \eqref{eq:dij}, we find that
$$
\begin{aligned}
\frac{d}{dt}\left( \frac{\Delta s_2}{s_1}\right)
&\ge-\log\lambda 2^{\a}\frac{\a\mu+\a+2}{r_0^{\a}}s_1^{2\a}\frac{\Delta s_2}{s_1}\\
&\hspace{0.5cm}-\frac{12\a \log\lambda}{r_0^{\a}}s_1^{2\a}(\frac{\xi_1^2 +\xi_2^2}{s_1^2})^{\alpha-\frac12}\Big(\frac{\Delta s_2}{s_1}\Big)^2.
\end{aligned}
$$ 
Let $\tilde\chi=\tilde\chi(t)=\frac{\Delta s_2}{s_1}(t)$. Then the above inequality can be written as
\begin{equation}\label{newchi1}
\frac{d\tilde\chi}{dt}\ge -\frac{\log\lambda}{r_0^\alpha}s_1^{2\alpha}\tilde\chi\Bigl(2^{\a}(\a\mu+\a+2)+12\alpha\left(\frac{\xi_1^2+\xi_2^2}{s_1^2}\Bigr)^{\alpha-\frac12}\tilde\chi\right).
\end{equation}
It is shown in \cite{THK} (see page 19) that
\begin{equation}\label{ineq1}
\Bigl(\frac{\xi_1^2+\xi_2^2}{s_1^2}\Bigr)^{\alpha-\frac12}\le 
\begin{cases}
(1-\tilde\chi)^{2\alpha-1} & 0<\alpha\le\frac12, \\
2^{\alpha-\frac12}(1+\tilde\chi)^{{2\alpha-1}} & \frac12\le\alpha<1 
\end{cases}
\end{equation}
and that $\tilde\chi=\frac{\Delta s_2}{s_2}$ is positive and decreasing (or negative and increasing). Observing that $s_1(T_1) =s_2(T_1)$ and using Assumption (2) of the lemma,  we obtain
$$ 
0\le\tilde\chi(T_1)=\frac{\Delta s_2(T_1)}{s_1(T_1)} = \frac{\Delta s_2(T_1)}{s_2(T_1)} 
\le\frac{\Delta s_2(0)}{s_2(0)} < \frac{1-\mu}{72}.
$$
\begin{equation}\label{ineq2}
\tilde{\chi}(t)=\frac{\Delta s_2(t)}{s_1(t)} \le  \frac{\Delta s_2(T_1)}{s_1(T_1)}
(1+2^{\a}C_1s_1^{2\a}(t)(t-T_1))^{-\beta}\le  \frac{\Delta s_2(T_1)}{s_1(T_1)}.
\end{equation}
Setting $B=2^{\a}(\a\mu+\a+2)+\frac{1-\mu}{3}\a$ and combining \eqref{newchi1}, \eqref{ineq1}, and \eqref{ineq2}, we obtain that
\begin{equation}\label{eq:chi-estimate1}
\frac{d\tilde\chi}{dt}\Big|_{t=T_1}\ge-\frac{B\log\lambda}{r_0^\alpha}s_1^{2\alpha}(T_1)\tilde\chi(T_1).
\end{equation}
Therefore, Gronwall's inequality and the fourth inequality in Lemma~\ref{lem5.2} now yield
\[  	
\begin{aligned}
\tilde\chi(t)&\ge\tilde\chi(T_1)\exp\left(-B\frac{\log\lambda}{r_0^\alpha} 
\int_{T_1}^t s_1^{2\alpha}(\tau)\,d\tau\right) \\
&\ge\tilde\chi(T_1)\exp\left(-B\frac{\log\lambda}{r_0^\alpha} 
\int_{T_1}^t s_1^{2\alpha}(t) (1+C_1 s_1^{2\alpha}(0)\,(t-\tau))^{-1}\,d\tau\right) \\
&=\tilde\chi(T_1)\exp\left(-B\frac{\log\lambda}{r_0^\alpha} 
\frac{1}{C_1}\log(1+C_1 s_1^{2\alpha}(t)\,(t-T_1)\ )\right)\\ 
&=\tilde\chi(T_1)\exp\left(-B\log(1+C_1s_1^{2\alpha}(t)\,(t-T_1)\ )\right)\\
&=\tilde\chi(T_1)\left(1+C_1s_1^{2\alpha}(t)\,(t-T_1)\ \right)^{-\beta_1},
\end{aligned}
\]
where $\beta_1=\frac{B}{2\alpha}$. It follows that 
$$
\Delta s_2(t)\ge\frac{\Delta s_2(T_1)}{s_1(T_1)}s_1(t)\left(1+C_1\, s_1^{2\alpha}(T_1)\,(t-T_1)\right)^{-\beta_1}, \,\,T_1 \le t \le T.
$$
In the case $s_1<0$ one can show using argument similar to the above that the same estimate for $\tilde\chi(t)$ holds but with exponent $2^{\alpha-1}(1+\mu)+\frac{2^\alpha}{\a}-\frac{1-\mu}{6}<\beta_1$. This completes the proof of the second estimate.
\end{proof}	
\begin{lem}\label{deltas2low2} 
Under the assumptions of Lemma \ref{deltas2upp} one has for some $C_4>0$ and all $0<t<T_1$ that	$\Delta s_2(t)\ge C_4\Delta s_2(0)t^{-\gamma}$.
\end{lem}
\begin{proof}
By assumption, Lemma \ref{deltas2low1} holds and the first estimate together with the first inequality in Lemma \ref{lem5.2} yield
$$
\begin{aligned}
\Delta s_2(t)&\ge \frac{\Delta s_2(0)}{s_2(0)}s_2(t)(1+C_1s_2^{2\a}(0)t)^{-\beta}\\
&\ge \Delta s_2(0)(1+C_1s_2^{2\a}(0)t)^{-(\frac{1}{2\a}+\beta)}.
\end{aligned}
$$
Since $1+C_1s_2^{2\a}(0)t\le(1+C_1s_2^{2\a}(0))t$ for $t\ge1$, the above implies
$$
\Delta s_2(t)> \Delta s_2(0)(1+C_1s_2^{2\a}(0))^{-(\frac{1}{2\a}+\beta)}t^{-(\frac{1}{2\a}+\beta)}.
$$
It remains to observe that $s_2(0)$ is of order $r_0$ and so 
$(1+C_1s_2^{2\a}(0))^{-(\frac{1}{2\a}+\beta)}$ is a constant and 
$\gamma=\frac{1}{2\a}+\beta$.

For $0<t\le 1$ the orbit stays bounded away from the region of perturbation and so the inequality holds for some constant.
\end{proof}	

\begin{lem}\label{deltas2low3}
Under the assumptions of Lemma \ref{deltas2upp} one has for some $C_5>0$ and all $T_1<t<T$ that $\Delta s_2(t)\ge C_5 \Delta s_2(T_1)$.
\end{lem}
\begin{proof}
The fifth inequality in Lemma \ref{lem5.2} implies that 
$2^{\a}C_1\, s_1^{2\alpha}(T_1)\,(t-T_1)<1$ and hence, 
$$
1-C_1\, s_1^{2\alpha}(T_1)\,(t-T_1)>1-\frac{1}{2^{\a}}.
$$	
In addition, since $s_1$ is increasing with respect to time $t$ we have $s_1(t)>s_1(T_1)$
By assumption, Lemma \ref{deltas2low1} holds and the second estimate of this lemma yields 
$$
\begin{aligned}
\Delta s_2(t)&\ge\frac{\Delta s_2(T_1)}{s_1(T_1)}s_1(t)\left(1-C_1\, s_1^{2\alpha}(T_1)\,(t-T_1)\right)^{\beta_1}\\
&> \Delta s_2(T_1)(1-\frac{1}{2^\a})^{\beta_1}.
\end{aligned}
$$ 	
This completes the proof of the lemma. 
\end{proof}
Consider a set $\Lambda^s_i$ and note that it consists of full length $s$-curves. Let us fix one of these curves, say $\sigma$.
\begin{lem}\label{ratioest}
Assume that $\sigma$ enters the slow down disk $D_{\frac{r_0}{2}}$ around the origin at time $n$, so that the intersection $f^n(\sigma)\cap D_{\frac{r_0}{2}}$ is not empty. Assume that $\sigma$ then exits $D_{\frac{r_0}{2}}$ at time $m$, $m>n>1$. Then
$$
C_6(m-n)^{-\gamma}\le\frac{L(f^m(\sigma))}{L(f^n(\sigma))}\le C_7(m-n)^{-\gamma'},
$$
where $C_6>0$, $C_7>0$, $\gamma$, $\gamma'$ are as in \eqref{gamma-gamma'}, and $L$ denotes the curve length. 
\end{lem}
\begin{proof}
Let $x$ and $y$ be the endpoints of the curve $\sigma$. For $k\ge 0$ set $x_k=f^k(x)$ and $y_k=f^k(y)$. It is easy to see that there is $K_0>0$ such that for all $k\ge 1$, 
\begin{equation}\label{ratioest1}
K_0^{-1}d(x_k, y_k)\le L(f^k(\sigma))\le K_0d(x_k, y_k),
\end{equation}
where $d$ denotes the usual distance.
 
Let $s,\tilde{s}: [0, N]\to R^2$ be the solutions of Equation \eqref{batata10} with initial conditions $s(0)=x_n$ and $\tilde{s}(0)=y_n$ respectively. Also, define 
$\Delta s_i(t)=\tilde{s_i}(t)-s_i(t)$, $i=1,2$ and $\Delta s=(\Delta s_1, \Delta s_2)$. Note that
there is $K_1>0$ such that for all $n,m>0$ and $n\le j\le m$
\begin{equation}\label{ratioest2}
K_1^{-1}||\Delta s(j)||\le d(x_j, y_j)\le K_1||\Delta s(j)||.
\end{equation}
We wish to apply Lemmas \ref{deltas2upp}--\ref{deltas2low3} to the orbits of $x_n$ and $y_n$ and we need to check the assumptions of Lemma \ref*{deltas2upp}. Assumption (1) is satisfied, since $y$ is contained in the stable cone at $x$. Assumption (2) requires 
$d(x_i, y_i), i=n_0,m_0$ to be sufficiently small. In view of Proposition 6.1 and \eqref{ratioest2}, this can be ensured, if we choose the number $r_0$ in the construction of the slow down map sufficiently small to guarantee that $Q$ is sufficiently large.
 
We have that $f^j(\sigma)\subset D_{\frac{r_0}2}\cap\{(s_1,s_2): s_2\ge s_1\}$ for 
$n\le j\le\frac{n+m}{2}$ and $f^j(\sigma)\subset D_{\frac{r_0}2}\cap\{(s_1,s_2): s_2< s_1\}$ for $\frac{n+m}{2}< j<m$. Applying Lemmas \ref{deltas2low2} and \ref{deltas2low3} with 
$0\le t\le\frac{n+m}{2}$ and $\frac{n+m}{2}< t\le m$ and using \eqref{ratioest1}, we obtain that
$$
\begin{aligned}
L(f^m(\sigma))&\ge K_0^{-1}d(x_m,y_m)\ge K_0^{-1}K_1^{-1}||\Delta s(m-n)||\\
&\ge K_0^{-1}K_1^{-1}\Delta s_2(m-n)>K_0^{-1}K_1^{-1}C_5\Delta s_2\big(\frac{m-n}{2} \big)\\
&\ge K_0^{-1}K_1^{-1}C_5C_4\Delta s_2(0)\bigl(\frac{m-n}{2}\bigr)^{-\gamma}\\
&\ge K_0^{-1}K_1^{-1}C_5C_42^{\gamma}(m-n)^{-\gamma}\frac{1}{\sqrt{1+\mu^2}}||\Delta s(0)||\\
&\ge K_0^{-2}K_1^{-2}C_5C_42^{\gamma}(m-n)^{-\gamma}\frac{1}{\sqrt{1+\mu^2}}L(f^n(\sigma)).
\end{aligned}
$$
Therefore, for some $C_6>0$,
$$
\frac{L(f^m(\sigma))}{L(f^n(\sigma))}\ge C_6(m-n)^{-\gamma}.
$$
Using again Lemmas \ref{deltas2upp1} and \ref{deltas2upp2} and inequalities \eqref{ratioest1} and \eqref{ratioest1}, similar arguments yield that for some $C_7>0$,
$$
\frac{L(f^m(\sigma))}{L(f^n(\sigma))}\le C_7(m-n)^{-\gamma'}.
$$
This completes the proof of the lemma.
\end{proof}

\section{Proof of Theorem \ref{mainthm1}: A lower bound for the tail of the return time}

In this section we establish a polynomial lower bound on the decay of the tail of the return time that is $m(\{x\in\Lambda: \tau(x)>n\})$. Consider the Markov partitions 
$\tilde{\mathcal{P}}$ and $\mathcal{P}$ for the automorphism $A$ and the map $f_{\T^2}$ respectively and let $\tilde{P}\in\tilde{\mathcal{P}}$ and $P\in\mathcal{P}$ be the elements of the partitions as in Section \ref{torustower}. Fix the number $Q$ as in Proposition \ref{f is Young}. We assume that the partition $\mathcal{P}$ and the number $r_0$ are chosen such that Proposition \ref{partition} holds and we set again $f:=f_{\T^2}$. Finally, we denote by 
$$
\mathcal{N}=\{n\in\mathbb{N}: \text{ there is } x\in P \text{ such that } n=\tau(x)\}.
$$ 
\begin{lem}\label{specialLambda_i^s}
There exists an integer $Q_1>0$ such that for any $N>0$ one can find $n>N$ with  $n\in\mathcal{N}$, an $s$-subset $\Lambda_\ell^s$ with $\tau(\Lambda_\ell^s)=n$ and numbers $0<m_1<m_2$ satisfying $m_1<Q_1,\ n-m_2<Q_1$ such that 
$f^k(\Lambda_\ell^s)\cap D^1_{r_0}=\emptyset$ for all $0\le k< m_1$ or $m_2< k\le n$ and $f^k(\Lambda_\ell^s)\cap D^1_{r_0}\ne\emptyset$ for all 
$m_1\le k\le m_2$.
\end{lem}
\begin{proof}
It suffices to show that there is $Q_1>0$ such that for any $N>0$ there is an admissible word of length $n>N$ with $n\in\mathcal{N}$ of the form
\begin{equation}\label{word}
P\bar{W}_1\bar{P_i}\bar{W}_2P,
\end{equation}
where the words $\bar{W}_1$ and $\bar{W}_2$ are of length $l(\bar{W}_j)< Q_1$ for $j=1,2$ and do not contain any of the symbols $P$ or $P_k$ (the element of the Markov partition containing $x_k$ for $k=1,2,3,4$), and the word $\bar{P_i}$ consists of the symbol $P_i$ which is repeated $n-2-l(\bar{W}_1)-l(\bar{W}_2)$ times. Since the map $f$ is topologically conjugate to $A$, it is enough to find an admissible word of the form \eqref{word} which consists of the corresponding elements of the partition $\tilde{\mathcal{P}}$. 

Note that $A=B^3$ where $B$ is an automorphism of the torus given by the matrix $B=\Bigl(\begin{matrix} 1&1\\1&2\end{matrix}\Bigr)$. Therefore the result would follow if we find an admissible word of the type of \eqref{word} for the automorphism $B$. To this end consider the stable and unstable separatrices through the origin and denote the ``first'' connected component of their intersection with $P$ by $\gamma^s$ and $\gamma^u$ respectively. It takes finitely many iterates of $G$ and $G^{-1}$ for each of these curves to completely enter the disk $D^1_{r_0}$. Now for each sufficiently large $n>0$ with $n\in\mathcal{N}$ there is an $s$-set $\Lambda^s_\ell$ with $\tau(\Lambda^s_\ell)=n$ which completely enters $D^1_{r_0}$ (under iterates of $G$ and $G^{-1}$) at the same time as 
$\gamma^s$ and $\gamma^u$ respectively. This completes the proof of the lemma.
\end{proof}

\begin{lem}\label{lowerbound}
There exists a constant $C_8>0$ such that
$$
m(\{x\in\Lambda: \tau(x)>n\})>C_8n^{-(\gamma-1)},
$$
where $\gamma$ is defined in \eqref{gamma-gamma'}. 
\end{lem}
\begin{proof}
Using the conjugacy \eqref{mapshi}, it suffices to prove the lemma for the map $G$. Write
$$
\begin{aligned}
m(\{x\in\Lambda: \tau(x)> n\})&=\sum\limits_{N=n+1}^{\infty}m(\{x\in\Lambda: \tau(x)=N\})\\
&=\sum\limits_{N=n+1}^{\infty}\sum\limits_{\Lambda_p^s: \tau(\Lambda_p^s)=N}m(\Lambda_p^s)>\sum\limits_{N=n+1}^{\infty}m(\Lambda_\ell^s),
\end{aligned}
$$
where $\Lambda_\ell^s$ is the set constructed in Lemma
\ref{specialLambda_i^s}. We wish to obtain a polynomial bound for the measure of the set 
$\Lambda_\ell^s$. 

Given $x\in\Lambda_\ell^s$, denote by $\gamma^s_\ell(x):=\gamma^s(x)\cap\Lambda_\ell^s$ (recall that $\gamma^s(x)$ is the full length stable curve through $x$ in the element $P$ of the Markov partition). There is $K_1>0$ such that 
\begin{equation}\label{est1}
m(\Lambda_\ell^s)=m(G^N(\Lambda_\ell^s)) =K_1L(G^N(\gamma^s_\ell(x))),
\end{equation}
where $L$ stands for the length of the curve.

Let $x_j=G^j(x)$ for $j=0,\dots, n$. Assume that $x$ enters the region $D^1_{r_0}$ at time $k_1$ and exits at time $k_2$, i.e.,
\begin{enumerate}
\item $G^j(x)\notin D^1_{r_0}$ if $0\le j<k_1,$ or $k_2<j\le N$;
\item $G^j(x)\in D^1_{r_0}$ if $k_1\le j\le k_2$.
\end{enumerate}
Note that for $0\le j<k_1$ and $k_2<j\le N$ the curve $G^j(\gamma^s_\ell(x))$ lies in the stable cone for the automorphism $A$ at $x_j$ and indeed, is an admissible manifold for $A$ (i.e., for any $y\in\gamma^s_\ell(x)$ the line $T_y\gamma^s_\ell(x)$ lies in the stable cone at $y$). So the length of the curve $\gamma^u_\ell(x)$ expands exponentially outside of the region $D_{r_0}.$ Since by Lemma \ref{specialLambda_i^s}, $k_1<Q_1$ and 
$N-k_2<Q_1$, we have that 
\begin{equation}\label{linest1}
L(\gamma^s_\ell(x))= \lambda^{k_1}L(G^{k_1}(\gamma^s_\ell(x)))\le \lambda^{Q_1}L(G^{k_1}(\gamma^s_\ell(x)))
\end{equation}
and
\begin{equation}\label{linest2}
L(G^{N}(\gamma^s_\ell(x)))=\lambda^{-(N-k_2)}L(G^{k_2}(\gamma^s_\ell(x)))\ge \lambda^{-Q_1}L(G^{k_2}(\gamma^s_\ell(x))),
\end{equation}
where  $\lambda$ is the largest eigenvalue of the matrix $A$. 

By Lemma 5.6 in \cite{THK}, the time the trajectory spends in $
D^1_{r_0}\setminus D^1_{\frac{r_0}{2}}$ is uniformly bounded. Thus, by Lemma \ref{ratioest},
\begin{equation}\label{lowratioest}
L(G^{k_2}(\gamma^s_\ell(x)))>C_6(k_2-k_1)^{-\gamma}L(G^{k_1}(\gamma^s_\ell(x))).
\end{equation}
Since $k_2-k_1<N$, combining Equations \eqref{est1}-\eqref{lowratioest} yields
$$
\begin{aligned}
m(\Lambda_\ell^s)
&\ge K_2L(G^N(\gamma^s_\ell(x)))=K_2\lambda^{-Q_1}L(G^{k_2}(\gamma^s_\ell(x)))\\
&\ge K_2C_6\lambda^{-Q_1}(k_2-k_1)^{-\gamma}L(G^{k_1}(\gamma^s_\ell(x)))\\
&\ge K_2C_6\lambda^{-2Q_1}(k_2-k_1)^{-\gamma}L(\gamma^s_\ell(x))
\ge K_3N^{-\gamma},
\end{aligned}
$$
where $K_2>0$ is a constant and $K_3= K_2C_6\lambda^{2Q_1}L(\gamma^u_\ell(x))$.

Note that $\gamma^s_\ell(x)$ is a full length stable curve in $P$ and hence, has length which is independent of $N$. It follows that 
$$
m(\{x\in\Lambda: \tau(x)>n\})>\sum\limits_{N=n+1}^{\infty}m(\Lambda_\ell^s)>C_8\frac{1}{n^{\gamma-1}},
$$
where $C_8>0$ is a constant. The desired lower bound follows.
\end{proof}

\section{Proof of Theorem \ref{mainthm1}: An upper bound for the tail of the return time}

In this section we obtain an upper polynomial bound for the decay of the tail of the return time. As before we assume that the Markov partition and the number $r_0$ are chosen such that Proposition \ref{partition} holds. Recall that $D_r$ is the union of the disks $D^i_r$ around the points $x_i$ and $P_i$ is the element of the partition containing $x_i$, $i=1,2,3,4$. We have that $D^i_{r_0}\subset P_i$. Using the conjugacy \eqref{mapshi}, it suffices to establish that upper bound for the map $G$. 

Given an $s$-set $\Lambda_i^s\subset P$ with $\tau(\Lambda_i^s)=n$, choose any numbers $k=k(\Lambda_i^s)$, $p=p(\Lambda_i^s)$, and two finite collections of numbers $\{k_m\ge 0\}_{m=1,\dots, p}$ and $\{l_m\ge 0\}_{l=0,\dots, p}$ such that 
\begin{enumerate}
%\item $l_0=0$;
\item $k_1+k_2+\cdots +k_p=k$ and $l_1+l_2+\cdots+l_{p+1}=n-k$;
\item the trajectory of the set $\Lambda_i^s$ under $G^j$, $0\le j\le n$, consecutively spends $l_m$-times outside $D_{r_0}$ and $k_m$-times inside $D_{r_0}$.
%, then enters D_{r_0} at time $l_1+1$ and stays there $k_1$-times, $D\cap D_{r_0}=\emptyset,\ k_m< j\le l_{m+1},\ m=1,2,\dots, p$;
%\item $G^j(\Lambda_i^s)\cap D_{r_0}\ne\emptyset,\ l_m\le j\le k_{m+1},\ m=0,2,\dots, p-1$,
\end{enumerate}
Given $0<p<k<n$, consider the collections
$$
\mathcal{S}_{k,n,p}=\{\Lambda_i^s\subset P: \tau(\Lambda_i^s)=n, \, k=k(\Lambda_i^s), \, p=p(\Lambda_i^s)\}. 
$$
\begin{lem}\label{countingLambdas}
There are $0<h<h_{\text{top}}(f)$, $\varepsilon_0>0$, and $C_9>0$ such that 
$\varepsilon_0<h_{\text{top}}(f)-h$ and  
$$
\text{Card}\ \mathcal{S}_{k,n,p}\le C_9\frac{1}{p^2}e^{(h+\varepsilon_0)(n-k)}.
$$ 
\end{lem}
\begin{proof}
Note that the cardinality of $\mathcal{S}_{k,n,p}$ does not exceed the number of symbolic words of length $n$ that start and end at $P$ and contain exactly $k$ symbols $P_j$, $j=1,2,3,4$. Since $A$ is topologically mixing, the latter is exactly the number of words of length $n-k$ that start and end at $P$. By Corollary 1.9.12 and Proposition 3.2.5 in \cite{KH}, the number of such words grows exponentially with an exponent that does not exceed $(n-k)h$ where $0<h<h_{\text{top}}(A)$.

The number of different ways the iterates of $\Lambda_i^s$ can enter $D_{r_0}$ exactly $p$ times and stay in this set exactly $k$ times does not exceed the number of ways in which the number $k$ can be written as a sum of $p$ positive integers (where order matters) which is equal to $\binom{k-1}{p-1}$. The number of different ways the iterates of 
$\Lambda_i^s$ can spend outside $D_{r_0}$ exactly $p+1$ times is equal to the number of ways in which the number $n-k$ can be written as a sum of $p+1$ positive integers which is $\binom{n-k-1}{p}$. Since iterates of $\Lambda_i^s$ may enter any of the disks 
$D^i_{r_0},\ i=1,2,3,4$, we obtain 
$$
\text{Card }\mathcal{S}_{k,n,p}\le C_94^p\binom{k-1}{p-1}\binom{n-k-1}{p}e^{h(n-k)}.
$$ 
Write
$$
\text{Card }\mathcal{S}_{k,n,p}\le \frac{C_9}{p^2}p^24^p\binom{k-1}{p-1}\binom{n-k-1}{p}e^{h(n-k)}.
$$ 
To prove the lemma we wish to estimate $p^24^p\binom{k-1}{p-1}\binom{n-k-1}{p}$ and we claim that there is $\varepsilon_0>0$ such that $\binom{k-1}{p-1} <e^{\varepsilon_0(n-k)}$.

To this end note that by Propositions \ref{partition} and \ref{f is Young}, it takes 
$\Lambda_i^s$ at least $Q$ iterates before it enters $D_{r_0}$ again. This implies that 
$n=k+l_1+\cdots+l_{p+1}>k+(p+1)Q$ that is $p+1<\frac{n-k}{Q}$.

For a fixed $k$ note that $\binom{k-1}{p-1}$ achieves its maximum when $p-1=[\frac{k-1}{2}]$ or  $p-1=[\frac{k-1}{2}]+1$. We may assume $p-1=\frac{k-1}{2}$. Then using the asymptotic formula $\binom{m}{l}\sim(\frac{me}{l})^l$, we obtain that
$$
\binom{k-1}{p-1}<\binom{2p-2}{p-1}<4^{p-1}=e^{(p-1)\ln 4}<e^{\frac{n-k}{Q}\ln 4}.
$$
To estimate $\binom{n-k-1}{p}$ observe that $p$ does not exceed $\frac{n-k}{Q}$. Hence, using the above asymptotic formula, we find that
$$
\begin{aligned}
\binom{n-k-1}{p}&<\binom{n-k}{\frac{n-k}{Q}}<
\Big(\frac{(n-k)e}{\frac{n-k}{Q}}\Big)^{\frac{n-k}{Q}}\\
&<e^{\frac{n-k}{Q}\ln\frac{(n-k)e}{\frac{n-k}{Q}}}<e^{\frac{n-k}{Q}\ln (Qe)}.
\end{aligned}
$$
Finally, note that 
$$
p^24^p<e^{2\ln p+p\ln4}<e^{2p+p\ln4}<e^{\frac{n-k}{Q}(\ln4+2)}.
$$
Now, given any sufficiently small $\varepsilon_0>0$, one can choose $Q$ large enough so that $\frac{\ln4+2}{Q}+\frac{\ln4}{Q}+\frac{\ln (Qe)}{Q}<\varepsilon_0$. Combining the above estimates we obtain
$$
4^p\binom{k-1}{p-1}\binom{n-k-1}{p}e^{h(n-k)}<e^{(n-k)\varepsilon_0}e^{h(n-k)}=e^{(n-k)\varepsilon_0+h}
$$
and hence,
$$
\text{Card}\ \mathcal{S}_{k,n,p}\le C_9\frac{1}{p^2}e^{(h+\varepsilon_0)(n-k)}.
$$ 
This completes the proof of the lemma.
\end{proof}

\begin{lem}\label{upper-est}
There exists $\varepsilon_0>0$ such that for any $\Lambda_i^s\in\mathcal{S}_{k,n,p}$, 
$$
m(\Lambda_i^s)\le C_{10}k^{-\gamma'}e^{(-\log\lambda+\varepsilon_0)(n-k)},
$$
where $C_{10}>0$ is a constant and $\gamma'$ is given by \eqref{gamma-gamma'}. 
\end{lem}
\begin{proof}
Note that by \eqref{est1}), 
$m(\Lambda_i^s)=m(G^n(\Lambda_i^s))= K_1L(G^n(\gamma^s_i(x)))$ and that the length of the backward iterates of $\gamma^s_i$ lying outside the region $D_{r_0}$ are stretched by the largest eigenvalue $\lambda$ of the matrix $A$. Note also that every time the iterates of $\gamma^s_i$ enter the region we have an upper estimate for its length according to Lemma \ref{ratioest} (note that we can apply this lemma in the region $D_{r_0}$ since by Lemma 5.6 in \cite{THK} the time spent in $D_{r_0}\setminus D_{\frac{r_0}{2}}$ is uniformly bounded). Thus,
$$
\begin{aligned}
m(\Lambda_i^s)&=m(G^n(\Lambda_i^s))= K_1L(G^n(\gamma^s_i(x)))\\
&=K_1\lambda^{-l_{p+1}}L(G^{n-l_{p+1}}(\gamma^s_i(x)))\\
&\le K_1C_7\lambda^{-l_{p+1}}k_p^{-\gamma'}L(G^{n-(l_1+k_1)}(\gamma^s_i(x)))\le\cdots\\
&\le K_1C_7^p\lambda^{-(l_{p+1}+\cdots+l_1)}k_p^{-\gamma'}k_{p-1}^{-\gamma'}\cdots k_1^{-\gamma'}L(\gamma^s_i(x)).
\end{aligned}
$$
Note that $\gamma^s_i(x)$ is a full length stable curve in $P$ and hence, has length independent of $n$. One can also assume that $k_i\ge 2$ by making $r_0$ smaller if necessary. This implies 
$$
k_1k_2\cdots k_p\ge k_{\max}2^{p-1}\ge k_{\max}p\ge\sum_{i=1}^{p}k_i=k,
$$
where $k_{\max}$ denotes the largest of $k_i's$.

In addition, $C_6^p=e^{p\ln C_6}<e^{\frac{n-k}{Q}\ln C_6}<e^{\varepsilon_0(n-k)}$ for sufficiently small $\varepsilon_0>0$ if one chooses $Q$ large. Therefore,
$$
m(\Lambda_i^s)<K_1e^{\varepsilon_0(n-k)}\lambda^{-(n-k)}k^{-\gamma'}<C_{10}k^{-\gamma'}e^{(-\log\lambda+\varepsilon_0)(n-k)}.
$$
This completes the proof of the lemma.
\end{proof}

\begin{lem}\label{upperbound}
There exists $C_{11}>0$ such that
$$
m(\{x\in \Lambda: \tau(x)>n\})<C_{11}n^{-(\gamma'-1)},
$$
see \eqref{gamma-gamma'} for the definition of $\gamma'$.
\end{lem}
Note that
$$
m(\{x\in \Lambda:\tau(x)=n\})
\le\sum\limits_{k=1}^n\sum\limits_{p=1}^k\max_{\Lambda_i^s\in\mathcal{S}_{k,n,p}}\{m(\Lambda_i^s)\}\text{Card }\mathcal{S}_{k,n,p}.
$$
Therefore, by Lemmas \ref{countingLambdas} and \ref{upper-est}, we have 
\begin{equation}\label{upmeasureest2}
\begin{aligned}
m(\{x\in\Lambda : \tau(x)&= n\})\\
&\le \sum\limits_{k=1}^{n}\sum_{p=1}^{k}\frac{1}{p^2}C_9e^{(h+\varepsilon_0)(n-k)}C_{10}e^{(\varepsilon_0-\log\lambda)(n-k)}k^{-\gamma'}\\
&<C_9C_{10}\frac{\pi^2}{6}e^{-\delta n}\sum\limits_{k=1}^{n}e^{\delta k}k^{-\gamma'},
\end{aligned}
\end{equation}
where $\delta=2\varepsilon_0+\log\lambda-h>0$ if $\varepsilon_0$ is sufficiently small.

To estimate $\sum\limits_{k=1}^{n}e^{\delta k}k^{-\gamma'}$ set 
$u_k=e^{\delta k}k^{-\gamma'}$ and note that 
$u_{k+1}-u_k\sim e^{\delta k}k^{-\gamma'}=u_k$. Since $\sum\limits_{k=1}^n u_k$ is positive and diverges, by Stolz-Cesaro theorem,
$$
\sum\limits_{k=1}^n u_k\sim \sum\limits_{k=1}^n u_{k+1}-u_{k}=u_{n+1}-u_1\sim e^{\delta n}n^{-\gamma'}.
$$
Therefore, 
$$
m(\{x\in \Lambda : \tau(x)= n\})\le C_9C_{10} e^{-\delta n}\sum\limits_{k=1}^{n}e^{\delta k}k^{-\gamma'}<C_9C_{10}n^{-\gamma'}.
$$
Thus, we have the following estimate of the tail 
$$
m(\{x\in \Lambda : \tau(x)>n\})=\sum\limits_{k> n}m(\{x\in \Lambda : \tau(x)= k\})< C_{11}n^{-(\gamma'-1)}
$$
for some $C_{11}>0$. This concludes the proof of the Lemma and the upper bound.

\section{Proof of Theorem \ref{mainthm1}: Carrying the slow-down map to a surface}

In this section we show how to carry over the slow-down map of the torus to a measure preserving diffeomorphism of any surface. Following \cite{katan}, we will construct the maps $\varphi_1,\varphi_2, \varphi_3$ such that the following diagram is commutative: 
\[
\begin{tikzcd}[column sep=3em, row sep=2em]
\T^2 \arrow{r}{\varphi_1} \arrow[swap]{d}{f_{\T^2}} & S^2\arrow{r}{\varphi_2} \arrow{d}{f_{S^2}} & D^2\arrow {r}{\varphi_3} \arrow{d}{f_{D^2}} & M\arrow{d}{f_M} \\%
\T^2 \arrow{r}{\varphi_1}& S^2\arrow{r}{\varphi_2} & D^2\arrow{r}{\varphi_3} &M
\end{tikzcd}
\]
We stress that while construction of the maps $\varphi_1$ and $\varphi_2$ follows \cite{katan}, construction of the map $\varphi_3$ is quite different, since we have to deal with finite regularity of the slow-down map.

First, using the slow down map we construct a diffeomorphism of the sphere $S^2$. 
\begin{prop}[see \cite{katan}]
There exists a map $\varphi_1\colon\T^2\to S^2$ satisfying:
\begin{enumerate}
\item $\varphi_1$ is a double branched covering, is one-to-one on each branch, and
$C^\infty$ everywhere except at the points $x_i$, $i=1,2,3,4$ where it branches;
\item $\varphi_1\circ I =\varphi_1$ where $I\colon \T^2\to \T^2$ is the involution map  given by $I(t_1,t_2)=(1-t_1, 1-t_2)$;
\item $\varphi_1$ preserves area, i.e., $(\varphi_1)_*m=m_{S^2}$ where $m_{S^2}$ is the area in $S^2$;
\item there exists a coordinate system in each disk $D^i_{r_0}$ such that 
\[
\varphi_1(s_1,s_2)=\left(\frac{{s_1}^2-{s_2}^2}{\sqrt{{s_1}^2+{s_2}^2}},
\frac{2s_1s_2}{\sqrt{{s_1}^2+{s_2}^2}}\right);
\]
\item  The map $f_{S^2}:=\varphi_1\circ f_{\T^2}\circ\varphi_1^{-1}$ preserves the area.
\end{enumerate}
\end{prop}
The sphere can be unfolded onto the unit disk $D^2$ and the map $f_{S^2}$ can be carried over to an  area preserving map $f_{D^2}$ of the disk which is identity on the boundary of the disk. To see this set $p_i=\varphi_1(x_i)$, $i=1,2,3,4$. In a small neighborhood of the point $p_4$ we define a map $\varphi_2$ by
\[
\varphi_2(\tau_1,\tau_2)=\left(\frac{\tau_1\sqrt{1-\tau_1^2-\tau_2^2}}
{\sqrt{\tau_1^2+\tau_2^2}}, \,\frac{\tau_2\sqrt{1-\tau_1^2-\tau_2^2}}
{\sqrt{\tau_1^2+\tau_2^2}}\right).
\]
One can extend $\varphi_2$ to an area preserving $C^\infty$ diffeomorphism (still denoted by $\varphi_2$) between $S^2\setminus\{p_4\}$ and the interior of the unit disk $D^2$. The map
\begin{equation}\label{kip}
f_{D^2}:=
\begin{cases}  \varphi_2\circ f_{S^2}\circ\varphi_2^{-1} \ \text{on}\ \text{int}D^2\\
Id \ \ \ \ \ \ \ \ \ \  \ \ \  \ \ \ \text{on}\  \partial D^2
\end{cases}
\end{equation}
is a diffeomorphism of $D^2$ that preserves area $m_{D^2}.$
\begin{prop}\label{smoothnessindisk}
The maps $f_{S^2}$ and $f_{D^2}$ are of class of smoothness $C^{2+2\kappa}$ where 
$\kappa=\frac{\a}{1-\a}$.
\end{prop}
\begin{proof}
Using the explicit local expressions for $f_{S^2}$ and $f_{D^2}$ and following arguments in Proposition \ref*{f is1plusbeta}, we find that the maps $f_{S^2}$ and $f_{D^2}$ are Hamiltonian with respect to the area and the  Hamiltonian functions are given as
$$
H_3(\tau_1,\tau_2)=\frac{\tau_2 h(\sqrt{\tau_1^2+\tau_2^2})}{\sqrt{\tau_1^2+\tau_2^2}}\log\lambda
$$
and
$$
H_4(x_1,x_2)=\frac{x_2 h(\sqrt{1-x_1^2-x_2^2})}{\sqrt{x_1^2+x_2^2}}\log\lambda
$$	
respectively. Here, as before, $h(u)=u^{\frac{2}{1-\a}}$.
	
To show that the maps $f_{S^2}$ and $f_{D^2}$ are of the desired class of smoothness, we will show that the Hamiltonian functions $H_3$ and $H_4$ have H\"older continuous second order partial derivatives with H\"older exponent $2\kappa$. Since $H_3$ and $H_4$ are of the same regularity we consider only one of them and we set
$g(x,y)=y(x^2+y^2)^{\delta}$ where $\delta=\frac{1}{1-\a}-\frac 12$. 
	
Obviously, $\frac{\partial g}{\partial x}(0,0)=0$ and
$$
\frac{\partial g}{\partial x}(x,y)=
\begin{cases}
2\delta xy(x^2+y^2)^{\delta-1},& (x,y)\neq(0,0);\\
0, & (x,y)=(0,0).
\end{cases}.
$$
The function $\frac{\partial g}{\partial x}$ is symmetric, so we will only study H\"older continuity of $\frac{\partial^2 g}{\partial x^2}$ as H\"older continuity of 
$\frac{\partial^2 g} {\partial x\partial y}$ is immediate. Note that
$$
\frac{\partial^2 g} {\partial x^2}(x,y)=
\begin{cases}
2\delta\frac{y(x^2+y^2)^{1-\delta}-2(1-\delta)x^2y(x^2+y^2)^{-\delta}}{(x^2+y^2)^{2-2\delta}}&  (x,y)\neq (0,0);\\
0 & (x,y)=(0,0). 
\end{cases}
$$ 
Since the function $\frac{\partial^2 g} {\partial x^2}(x,y)$ is differentiable for all 
$(x,y)\ne (0,0)$, it is H\"older continuous for all pairs of nonzero points $(x,y)$. It remains to show H\"older continuity for pairs of points one of which is zero. We can write  
$$
\frac{\partial^2 g}{\partial x^2}=K_1y(x^2+y^2)^{\delta-1}-K_2x^2y(x^2+y^2)^{\delta-2},
$$ 
where $K_1>0$ and $K_2>0$ are some constants. Choose $(x,y)\neq (0,0)$ and note that 
$$
|y(x^2+y^2)^{\delta-1}|\le (x^2+y^2)^{\frac 12}(x^2+y^2)^{\delta-1}=(x^2+y^2)^{\delta-\frac 12}=d((x,y),(0,0))^{2\delta-1}. 
$$
Similarly,
$$
|x^2y(x^2+y^2)^{\delta-2}|\le d((x,y),(0,0))^{2\delta-1}.
$$
Hence, $\frac{\partial^2 g} {\partial x^2}$ is H\"older continuous with H\"older exponent 
$$
2\delta-1=2\Bigl(\frac{1}{1-\a}-\frac 12\Bigr)-1=2\kappa.
$$
Further, each of the functions $\frac{\partial^2 g}{\partial y^2}$ and 
$\frac{\partial^2 g}{\partial y\partial x}$ can be written as a linear combinations of functions 
\begin{equation}\label{funts}
x(x^2+y^2)^{\delta-1}, \quad y(x^2+y^2)^{\delta-1}, \quad x^2y(x^2+y^2)^{\delta-2}, \quad y^2x(x^2+y^2)^{\delta-2}
\end{equation} 
and are $0$ at the point $(x,y)=(0,0)$. Arguing as above one can show that each function in \eqref{funts} is H\"older continuous with H\"older exponent $2\kappa$. This completes the proof of the proposition.
\end{proof}	

Consider a smooth compact connected oriented surface $M$. It can be cut along closed geodesics in such a way that the resulting surface with boundary is homeomorphic to a regular polygon via a homeomorphism which we denote by $T$; this is a well-known topological construction. 

Let now $f$ be a $C^\infty$ diffeomorphism of the disk, which is identity on the boundary 
$\partial D^2$ and is \emph{infinitely flat}, i.e., given a sequence $\rho_n\to 0$ and a sequence of open domains $V_n\subset D^2$ satisfying 
\begin{equation}\label{seq1}
V_n\subset\overline{V}_n\subset V_{n+1} \text{ and } \bigcup_{n\ge 1}V_n=D^2,
\end{equation}
we have that for every $n\ge 1$,
$$
V_{n-1}\subset f(V_n)\subset V_{n+1}
\ \ \text{and} \ \  
\|f - \mbox{Id}\|_{C^n(V_{n+1}\setminus V_{n-1})}\le\rho_n.
$$
For such an $f$ it is shown in \cite{katan} (see also \cite{BP13}) that there is a homeomorphism $\varphi:\overline{D^2}\to M$ such that 
\begin{enumerate}
\item $\varphi$ is of class $C^\infty$ in the interior of the disk; 
\item $\varphi$ is area preserving, i.e., $h_*m_{D^2}=m_M$; 
\item the map $\varphi\circ f\circ \varphi^{-1}$ is a $C^\infty$ area preserving diffeomorphism of the surface.
\end{enumerate} 
In our case however, the map $f=f_{D^2}$ is only of class $C^{2+2\kappa}$ and hence, is only \emph{finitely flat} at the boundary, i.e., there is a sequence of open domains $V_n\subset D^2$, satisfying \eqref{seq1}, such that for every $n\ge 1$ we have $V_{n-1}\subset f_{D^2}(V_n)\subset V_{n+1}$ and for every $0<\beta<2+2\kappa$,
\begin{equation}\label{sequ2}
\|f_{D^2}-\text{Id}\|_{C^{1+\beta}(V_{n+1}\setminus V_{n-1})}\le (r_{n-1})^{2+2\kappa-\beta},
\end{equation}
where $r_n=\text{dist}(V_n,\partial D^2)$. This requires us to develop a specific construction of the homeomorphism $\varphi$ which guarantees that the map $f_M$ is an area preserving diffeomorphism of class $C^{1+\beta}$ for some $\beta>0$. More precisely, the following statement holds.

\begin{thm}\label{mapphi3}
Given a smooth compact connected oriented surface $M$ and numbers 
$\frac19<\a<\frac14$ and $0<\mu<\frac12$, there exist $\beta=\beta(\a,\mu)>0$ and a continuous map $\varphi_3\colon\overline{D^2}\to M$ such that
\begin{enumerate}
\item the restriction $\varphi_3|\text{int }D^2$ is a diffeomorphic embedding;
\item $\varphi_3(\overline{D^2}) = M$;
\item $\varphi_3$ preserves area; more precisely, $(\varphi_3)_*m_{D^2}=m_M$ where $m_M$ is the area in $M$; moreover, $m_M(M\setminus \varphi_3(\text{int }D^2))=0$;
\item the map $f_M:=\varphi_3\circ f_{D^2}\circ\varphi_3^{-1}$ is a $C^{1+\beta}$ area preserving diffeomorphism of the surface. 
\end{enumerate}
\end{thm}
One can represent a compact smooth oriented surface $M$ as a regular $p$-polygon $P$ (the number $p$ is even) whose angles are $\alpha=\frac{\pi(p-2)}{p}$. Let 
$A_1, A_2,\dots, A_p$ be vertices of the polygon and $O$ its center. For each 
$i=2,\dots, p$ denote by $B_i, \tilde{B}_i$ the points on the segment $A_iO$ for which 
$\frac{|A_iB_i|}{|A_iO|}=\frac 13$ and $\frac{|\tilde{B}_iO|}{|A_iO|}=\frac 13$. In what follows we assume that $A_{p+1}=A_1$, $B_{p+1}=B_1$, and $\tilde{B}_{p+1}=A_1$. 
Denote by
\begin{equation}\label{eq-p}
P^*:=A_1O\ \cup\ \Big(\bigcup_{i=1}^p A_iA_{i+1}\Big)\ \cup \Big(\bigcup_{i=2}^p (A_iB_i\cup\tilde{B}_iO)\Big).
\end{equation}
Note that the complement to $P^*$ is an open simply connected set. 

We now construct a homeomorphism from the unit disk $D^2$ onto $P$.
\begin{prop}
There exist 
\begin{enumerate}
\item a nested sequence of open simply connected sets 
$U_0\subset U_1\subset\cdots\subset U_n\subset\cdots$ satisfying 
$\bigcup_n U_n=P\setminus P^*$; 
\item a sequence of $C^{\infty}$ diffeomorphisms $h_n:U_n\to U_{n+1}$ for $n\ge 0$;
\item a number $\beta>0$
\end{enumerate} 
such that setting $h(x)=\lim\limits_{n\to\infty}h_{n-1}\circ\cdots \circ h_1\circ h_0$, we have 
that the map $f_P:P\to P$ given by
$$
f_P=
\begin{cases}
(h\circ f_{D^2}\circ h^{-1})(x),\ \ x\in P\setminus P^*,\\
\text{Id}\ \ \ \ \ \ \ \ \ \ \ \ \ \ \ \ \ \ \ \ \ \ \  \  \text{otherwise}
\end{cases}
$$	
is a $C^{1+\beta}$ diffeomorphism.
\end{prop}	

\emph{Proof of the proposition.} We split the proof into three steps. 

\vskip0.05in
{\bf Step 1.} We first construct a sequence of open sets $U_n$.

Fix $n>0$, $t\in[n,n+1]$ and let $r(t)>0$ be a strictly monotonically decreasing continuous function on $[1,\infty)$ which will be determined later in Step 4. For $i=1,\dots, p$ consider the following collection of points associated with the point $A_i$:
\begin{itemize}
\item $K_{ti}$, the point that is determined uniquely by the requirements that 
the angle $\angle (K_{ti}A_iA_{i+1})=\frac18\alpha$ and $\text{dist}(K_{ti},A_iA_{i+1})=r(t)$;
\item $L_{ti}$, the point that is determined uniquely by the requirements that 
the angle $\angle (OA_iL_{ti})=\frac18\alpha$ and $\text{dist}(L_{ti},A_iO)=r(t)$;
\item $M_{ti}$, the image of $L_{ti}$ under the reflection about the line $OA_i$;
\item $N_{ti}$, the image of $K_{ti}$ under the reflection about the line $OA_i$;
\item $E_{ti}$, the point on the line through $B_i$ which is perpendicular to the line $A_iB_i$ and such that $\text{dist}(E_{ti},B_i)=r(t)$;
\item $F_{ti}$, the image of $E_{ti}$ under the reflection about the line $A_iB_i$.
\item $\tilde{L}_{ti}$, the point that is determined uniquely by the requirements that 
the angle $\angle (A_iO\tilde{L}_{ti})=\frac14(\pi-\alpha)$ and 
$\text{dist}(\tilde{L}_{ti},O\tilde{B}_i)=r(t)$;
\item $\tilde{M}_{ti}$, the image of the point $\tilde{L}_{ti}$ under the reflection about the line $A_iO$;
\item $\tilde{E}_{ti}$, the point on the line through $\tilde{B}_i$ which is perpendicular to the line $O\tilde{B}_i$ and such that $\text{dist}(\tilde{E}_{ti},\tilde{B}_i)=r(t)$;
\item $\tilde{F}_{ti}$, the image of the point $\tilde{E}_{ti}$ under the reflection about the line $O\tilde{B}_i$.
\end{itemize}
We introduce the following curves: for $i=1,\dots, p$ let
\begin{itemize}
\item $\gamma_{ti}^{(1)}$ be a line segment connecting the points $K_{ti}$ and $N_{t(i+1)}$ where we assume that $N_{t(p+1)}=N_{t1}$; 
\item $\gamma_{ti}^{(2)}$ be a curve connecting the points $K_{ti}$ and $L_{ti}$ to be determined later in Step 2;
\item $\gamma_{ti}^{(3)}$ be a curve connecting the points $M_{ti}$ and $N_{ti}$ to be determined later in Step 2;
\item $\tilde{\gamma}_{ti}^{(2)}$ be a curve connecting the points $\tilde{L}_{ti}$ and 
$\tilde{M}_{t(i+1)}$ to be determined later  in Step 2 where we assume that 
$\tilde{M}_{t(p+1)}=\tilde{M}_{t1}$.
\end{itemize}
Now for $i=2,\dots, p$ let
\begin{itemize}
\item $\gamma_{ti}^{(4)}$ be a line segment connecting the points $L_{ti}$ and $E_{ti}$;
\item $\gamma_{ti}^{(5)}$ be a line segment connecting the points $M_{ti}$ and $F_{ti}$;
\item $\gamma_{ti}^{(6)}$ be a curve connecting the points $E_{ti}$ and $F_{ti}$ to be determined later in Step 2; 
\item $\tilde{\gamma}_{ti}^{(4)}$ be a line segment connecting the points $\tilde{L}_{ti}$ and $\tilde{E}_{ti}$;
\item $\tilde{\gamma}_{ti}^{(5)}$ be a line segment connecting the points $\tilde{M}_{ti}$ and $\tilde{F}_{ti}$;
\item $\tilde{\gamma}_{ti}^{(6)}$ be a curve connecting the points $\tilde{E}_{ti}$ and 
$\tilde{F}_{ti}$ to be determined later in Step 2.
\end{itemize}
Finally, we let $\sigma^{(1)}_t$ to be a line segment connecting the points $L_{t1}$ and 
$\tilde{L}_{t1}$ and $\sigma^{(2)}_t$ to be a line segment connecting the points $M_{t1}$ and $\tilde{M}_{tp}$.

Let $\tau_t$ be the curve given by
$$
\tau_t=\sigma^{(1)}_t\cup\sigma^{(2)}_t\cup\bigl(\bigcup_{i=1}^p\gamma_{ti}^{(1)}\bigr)
\cup\bigr(\bigcup_{i=2}^p\bigr(\bigcup_{j=2}^6(\gamma_{ti}^{(j)}\cup\tilde{\gamma}_{ti}^{(j)})\bigl)\bigl).
$$
By construction, $\tau_t$, $t\in[n,n+1]$, is a closed connected continuous curve which bounds an open simply connected domain in the polygon $P$. We denote this domain by $U_t$. In particular, $U_n$ is the desired open set. 
\vskip0.05in
{\bf Step 2.} We show how to choose the curves $\gamma_{ti}^{(2)}$ and 
$\gamma_{ti}^{(6)}$. The curves $\gamma_{ti}^{(3)}$, $\tilde{\gamma}_{ti}^{(2)}$ and 
$\tilde{\gamma}_{ti}^{(6)}$ can be chosen in a similar way. 

Let $\varphi,\psi:[-1,1]\to\mathbb{R}$ be two continuous functions satisfying:
\begin{enumerate}
\item $\varphi\in C^\infty$ on $[0,1]$ and $\psi\in C^\infty$ on $(0,1)$;
\item $\varphi(-x)=\varphi(x)$ and $\psi(-x)=\psi(x)$;
\item $\varphi(0)=a$ where $1-\tan\frac{\alpha}{4}<a<\cot\frac{\alpha}{4}$ and 
$\varphi(-1)=\varphi(1)=\cot\frac{\alpha}{8}$;
\item $\psi(0)=1$ and $\psi(-1)=\psi(1)=0$;
\item $0<\varphi'(x)\le\cot\frac{\alpha}{4}$ for $0<x\le 1$ and $-\cot\frac{\alpha}{4}\le\varphi'(x)<0$ for $-1\le x<0$;
\item $\varphi'(-1)=-\cot\frac{\alpha}{4}$ and $\varphi'(1)=\cot\frac{\alpha}{4}$;
\item $\psi$ is infinitely vertically flat at $-1$ and $1$.
\end{enumerate}

For $i=1,\dots,p$ consider the orthogonal coordinate system with origin at $A_i$ whose vertical axis is the bisector of the angle $\angle (B_iA_iA_{i+1})$ and for every $t\in[n,n+1]$ we let $\gamma_{ti}^{(2)}$ be the graph of the function 
$\varphi_t(x)=r(t)\varphi(\frac{x}{r(t)})$ where $-r(t)\le x\le r(t)$. It is easy to see that $\gamma_{ti}^{(2)}$ is a $C^\infty$ curve that connects the points $K_{ti}$ and $L_{ti}$ and is infinitely tangent to the lines $K_{ti}N_{ti}$ and $L_{ti}E_{ti}$.

Now consider the orthogonal coordinate system with origin at $B_i$ whose vertical axis is the line $A_iB_i$. We let $\gamma_{ti}^{(6)}$ be the graph of the function 
$\psi_t(x)=r(t)\psi(\frac{x}{r(t)})$ where $-r(t)\le x\le r(t)$. It is easy to see that $\gamma_{ti}^{(6)}$ is a $C^\infty$ curve that connects the points $E_{ti}$ and $F_{ti}$ and is infinitely tangent to the lines $L_{ti}E_{ti}$ and $M_{ti}F_{ti}$.

It is easy to see that with the above choice of curves $\gamma_{ti}^{j}$, $j=1,\dots,6$, the curve $\tau_t$ is of class $C^\infty$.

We show that the curves $\tau_t$ corresponding to different values of $t$ are disjoint. To this end fix $n\le t_1<t_2\le n+1$. It suffices to show that the curves $\gamma_{ti}^{(2)}$, 
$\gamma_{ti}^{(6)}$, $\gamma_{ti}^{(3)}$, $\tilde{\gamma}_{ti}^{(2)}$, and 
$\tilde{\gamma}_{ti}^{(6)}$ with $t=t_1$ and $t=t_2$ are disjoint. We will prove this for the curve $\gamma_{ti}^{(2)}$ only as the proof for other curves is similar. By Property (3), 
$$
\varphi_{t_1}(0)=ar(t_1), \quad \varphi_{t_1}(r(t_1))=r(t_1)\cot\frac{\alpha}{8}, \quad 
\varphi_{t_1}'(r(t_1))=\cot\frac{\alpha}{4}.
$$
Similarly, 
$$
\varphi_{t_2}(0)=ar(t_2), \quad \varphi_{t_2}(r(t_2))=r(t_2)\cot\frac{\alpha}{8}, \quad 
\varphi_{t_2}'(r(t_2))=\cot\frac{\alpha}{4}.
$$
In view of Property (5) the desired result would follow if we show that 
$$
\varphi_{t_1}(r(t_2))=r(t_1)\varphi\Bigr(\frac{r(t_2)}{r(t_1)}\Bigl)
\ge\varphi_{t_2}(r(t_2))=r(t_2)\cot\frac{\alpha}{8}.
$$
Setting $x=\frac{r(t_2)}{r(t_1)}$, the above inequality amounts to
$\varphi(x)\ge x\cot\frac{\alpha}{8}$ and immediately follows from Properties (3) and (4) of the function $\varphi$. 
\vskip0.05in
{\bf Step 3.} We now construct maps $h_n$. By the Riemann Mapping theorem, there is a $C^\infty$ diffeomorphism $h_0:D^2\to U_1$. For each $n=1, 2,3,\dots$ we will construct maps $h_n:U_n\to U_{n+1}$ such that $h_n\big|{U_{n-1}}=\text{Id}$. 

Given two numbers $n-1\le s\le n$ and $n-1\le t\le n+1$ such that $s<t$ we construct a $C^\infty$ diffeomorphism $\hat{h}_{st}:\tau_s\to\tau_t$ in the following way. 
\begin{itemize}
\item $\hat{h}_{st}:\sigma^{(j)}_s\to\sigma^{(j)}_t$ is a linear map, given by 
$\hat{h}_{st}(z)=\frac{r(t)}{r(s)}(z)$, $z\in\sigma^{(j)}_s$ and $j=1,2$;
\item $\hat{h}_{st}:\gamma_{si}^{(1)}\to\gamma_{ti}^{(1)}$ is a linear map, given by 
$\hat{h}_{st}(z)=\frac{r(t)}{r(s)}(z)$, $z\in\gamma_{si}^{(1)}$ and $i=1,\dots,p$;
\item $\hat{h}_{st}:\gamma_{si}^{(j)}\to\gamma_{ti}^{(j)}$ is a map, given by 
$\hat{h}_{st}(z)=v$, where $z=(y,\varphi_s(y))$ and $v=(\frac{r(t)}{r(s)}y, \varphi_t(y))$ for 
$-r(s)\le y\le r(s)$, $i=2,\dots,p$, $j=2,\dots,6$;
\item $\hat{h}_{st}:\tilde{\gamma}_{si}^{(j)}\to\tilde{\gamma}_{ti}^{(j)}$ is a map, given by 
$\hat{h}_{st}(z)=v$, where $z=(y,\psi_s(y))$ and $v=(\frac{r(t)}{r(s)}y, \psi_t(y))$ for 
$-r(s)\le y\le r(s)$, $i=2,\dots,p$, $j=2,\dots,6$;
\end{itemize}
Now given $n-1\le s\le n$, define the map $\hat{h}_s:=\hat{h}_{st}$ with $t=2(s-n+1)+n-1$. The desired map $h_n:U_n\to U_{n+1}$ is now given as follows: for $A\in U_n$ choose a unique $s$ such that $A\in\tau_s$ with $n-1\le s\le n$ and then set $h_n(A)=B$ where  $B=\hat{h}_s(A)\in\tau_{2(s-n+1)+n-1}\subset U_{n+1}$. It is easy to see that $h_n$ is a $C^\infty$ diffeomorphism. 

It follows that the map $h=\lim\limits_{n\to\infty}h_{n-1}\circ\cdots\circ h_1\circ h_0$ is a well defined $C^\infty$ diffeomorphism from $\text{int}D^2$ onto $P\setminus P^*$ where $P^*$ is given by \eqref{eq-p}. It also follows from the construction of the map $h$ that there is $C>0$ such that 
\begin{equation}\label{h-diff}
\|h\|_{C^1}\le C, \quad \|h^{-1}\|_{C^1}\le C.
\end{equation}
\vskip0.05in

{\bf Step 4.} It remains to show that the map $f_P=h\circ f_{D^2}\circ h^{-1}$ is a 
$C^{1+\beta}$ diffeomorphism for some $\beta>0$. 

Observe that $f_P(P\setminus P^*)=P\setminus P^*$,  $f_P(P^*)=P^*$, and 
$f_P|P^*=\text{Id}$. In particular, $f_P|P\setminus P^*$ is a $C^\infty$ diffeomorphism. It remains to show that $f_P$ is of class $C^{1+\beta}$ on $P^*$. 
To do so we will show the following:

\begin{equation}\label{final}
\|f_P-\text{Id}\|_{C^{1+\beta}(U_{n+1}\setminus U_n)}\to 0
\end{equation} 
as $n\to\infty$.

First, we will prove the following lemma.
\begin{lem}
Let $r(n)$ be a decreasing sequence such that $0<r(1)< 1$ and 
\begin{equation}\label{rn}
r(n+1)= r^2(n), \quad 0<r(1)< 1.
\end{equation}
Then the sequence of open sets $V_n=h^{-1}(U_n)$ satisfies \eqref{seq1} and
$$
V_{n-1} \subset f_{D^2}(V_n)\subset V_{n+ 1}.
$$
\end{lem}
\textit{Proof of the Lemma}. Note that there are $C_2\ge C_1>0$ such that
$$
C_1r(n)\le\text{dist}(U_n, P^*)\le C_2r(n).
$$
Furthermore, in view of \eqref{h-diff} there are $C_4\ge C_3>0$ such that 
$$
C_3r(n)\le\text{dist}(V_n, \partial D^2)\le C_4r(n).
$$ 
Since the map $f_{D^2}$ is identity on $\partial D^2$ and is of class of smoothness $2+2\kappa$, we obtain that for all sufficiently small $r_n$, any $x$ in the neighborhood $U_{r_n}(\partial D^2)$ and any $\beta>0$
$$
\text{dist}(x, f_{D^2}(x)) < r(n)^{2+2\kappa-\beta}.
$$
Therefore, for some $0<a<1$,
$$
C_3r(n)- r(n)^{2+a}\le\text{dist}(f_{D^2}(x), \partial D^2)\le C_4r(n) + r(n)^{2+a}.
$$
To prove the desired inclusion we will show that
$$
C_4r(n+2)<C_3r(n)- r(n)^{2+a}\le C_4r(n) + r(n)^{2+a}< C_3r(n-2).
$$
We prove the leftmost inequality. Since $r(n+2)=r^4(n)$, we have
$$
\begin{aligned}
C_4r(n+2)&<C_3r(n)- r(n)^{2+a} \Leftrightarrow \\
C_4\frac{r(n+2)}{r(n)}&<C_3- r(n)^{1+a} \Leftrightarrow \\
C_4r(n)^3 &< C_3-r(n)^{1+a}. 
\end{aligned}
$$
For large values of $n$ both $C_4r(n)^3$ and $r(n)^{1+a}$ are small. Hence, the last inequality holds.

We now prove the rightmost inequality (the inequality in the middle is obvious).
$$
\begin{aligned}
C_4r(n) + r(n)^{2+a}&< C_3r(n-2) \Leftrightarrow \\
C_4+ r(n)^{1+a} &< C_3\frac{r(n-2)}{r(n)} \Leftrightarrow \\
C_4 + r(n)^{1+a}&< C_3r(n)^{-\frac 34}.
\end{aligned}
$$
For large values of $n$ the left hand side of the last inequality is close to $C_4$ while the right hand side gets large. This completes the proof of the lemma. 

The above lemma allows us to write
\begin{equation}\label{fP norm}
\begin{aligned}
\|h\circ(f_{D^2}-&\text{Id})\circ h^{-1}\|_{C^{1+\beta}(U_{n+1}\setminus U_n)}\le\\
\|h\|_{C^{1+\beta}(V_{n+2}\setminus V_{n-1})}&\|f_{D^2}-\text{Id}\|_{C^{1+\beta}(V_{n+1}\setminus V_n)}\|h^{-1}\|_{C^{1+\beta}(U_{n+1}\setminus U_n)}.
\end{aligned}
\end{equation}
Observe that $\|f_{D^2}-\text{Id}\|_{C^{1+\beta}}$ admits Estimate \eqref{sequ2} and it remains to estimate the norms $\|h\|_{C^{1+\beta}(V_{n+2}\setminus V_{n-1})}$
and $\|h^{-1}\|_{C^{1+\beta}(U_{n+1}\setminus U_n)}$. 

We write $h|_{V_{n+2}}= h_{n+1}|_{U_{n+1}},$ so in order to estimate the norm of $h$ we will estimate the norm of $h_{n+1}$. In order to estimate the norm of $h^{-1}$ we will need to estimate the norms of $h_n^{-1}$ for each $n = 0,1,2,\cdots$.

Further, it suffices to estimate the norm of $h$ restricted to the boundary of the sets $U_n$. Recall that $\tau_n$, the boundary of $U_n$, is a union of the curves $\sigma^{(1)}_n$, 
$\sigma^{(2)}_n$, $\gamma_{ni}^{(j)}$, and $\tilde{\gamma}_{ni}^{(j)}$, $j=1,\dots,6$. Note that the map $h$ acts linearly on the curves $\sigma^{(1)}_n$, $\sigma^{(2)}_n$, and 
$\gamma_{ni}^{(1)}$ and hence, the norm of $h$ restricted to these parts of the curve 
$\tau_n$ are bounded. 

The curves $\gamma_{ni}^{(j)}$ and $\tilde{\gamma}_{ni}^{j}$, $j\ge 2$ are the graphs of the functions $\varphi_n$ and $\psi_n$ respectively. We shall only give an estimate of the norm of $h$ restricted to the curves $\gamma_{ni}^{(j)}$, since the estimates of the norm of $h$ restricted to the curves $\tilde{\gamma}_{ni}^{(j)}$ is similar. Also, we can assume that $i$ and $j>1$ are fixed.

Now for a fixed curve $\gamma_{ni}^{(j)}$ we define the orthogonal coordinate system centered at the vertex $A_i$ with the vertical axis $A_iO$ (recall that $O$ is the center of the polygon $P$). In this coordinate  system the map $h_n$ is given by
$$
h_n:(x,\varphi_n(x))\rightarrow \Bigl(\frac{r(n+1)}{r(n)}x, \varphi_{n+1}(x)\Bigr), \,\, 
-r(n)\le x \le r(n).
$$
Since $\varphi_n$ is symmetric, we can further assume that $x>0$. We have that  
$$
\|h_n\|_{C^{1+\beta}}=\max\Bigl(\|h_n\|_{C^0}, \|dh_n\|_{C^0}, \|dh_n\|_{C^\beta}\Bigr),
$$
where 
$$
\|dh_n\|_{C^\beta}=\max\Bigl(\sup\frac{\|\partial_xh_n(x)-\partial_yh_n(y)\|}{\|x-y\|^{\beta}}, \sup\frac{\|\partial_xh_n(x)-\partial_yh_n(y)\|}{\|x-y\|^{\beta}}\Bigr)
$$
and $\partial_xh_n$ and $\partial_yh_n$ are the partial derivatives of $h$ with respect to the first and the second variables respectively.

Let us write $h_n= (h^{(1)}_n, h^{(2)}_n)$, $y=\varphi_n(x)$ where 
$$
h^{(1)}_n(x, y)=\frac{r(n+1)}{r(n)}x \,\text{ and }\, h^{(2)}_n(x, y) = \varphi_{n+1}(x).
$$
Since 
$$
h^{(1)}_n(x,y)=\frac{r(n+1)}{r(n)}x\le r(n+1)< 1 \,\text{ and }\, h^{(2)}_n(x,y)=\varphi_{n+1}(x) < K_1,
$$ 
we obtain that $\sup\|h_n\|< K_1$.

Further, we have that $\partial_x h^{(1)}_n(x,y)=\frac{r(n+1)}{r(n)}<1$ and  
$$
\begin{aligned}
\partial_y h^{(2)}_n(x,y)&=\frac{\partial}{\partial y}\Bigl(\frac{r(n+1)}{r(n)}x\Bigr)
=\frac{\partial}{\partial y}\Bigl(\frac{r(n+1)}{r(n)}\varphi_n^{-1}(y)\Bigr)\\
&= r(n+1)(\varphi^{-1})'\Bigl(\frac{x}{r(n)}\Bigr)\frac 1{r(n)} < 1.
\end{aligned}
$$
Also, $\partial h^{(2)}_n(x,y)=\varphi'_{n+1}(x)<\cot\frac{\alpha}{4}$ and 
$$
\begin{aligned}
\partial h^{(2)}_n(x,y)&=\frac{\partial }{\partial y}\Bigl(r(n+1)\varphi\Bigl(\frac{r(n)\varphi^{-1}(\frac{y}{r(n)})}{r(n+1)}\Bigr)\Bigr)\\
&=r(n+1)\varphi'\frac{r(n)}{r(n+1)}\varphi^{-1}\Bigl(\frac{y}{r(n)}\Bigr)(\varphi^{-1})'\frac 1{r(n)} < K_2. 
\end{aligned}
$$
Thus, the partial derivatives of the functions $h_1$ and $h_2$ are bounded. However, a similar calculation shows that $\|dh_n\|_{C^\beta}$ tends to infinity as $r(n)^{-\beta}$.   Therefore, we conclude that 
\begin{equation}\label{norm1}
\|h\|_{C^{1+\beta}(V_{n+2}\setminus V_{n-1})}\le r(n+1)^{-\beta} = r(n)^{-2\beta}.
\end{equation}
Similar computations holds for $h_n^{-1}: U_{n+1}\rightarrow U_{n},$ with the only difference that the partial derivatives estimated by $\frac{r(n)}{r(n+1)}$ which is unbounded. Therefore, the \holder norm of the first derivatives of $h_n^{-1}$ are bounded by 
$$
\frac{r(n)}{r(n+1)}\frac{1}{r(n)^{-\beta}}= \frac{r(n)^{1-\beta}}{r(n+1)}.
$$
It follows that 
\begin{equation}\label{norm2}
\begin{aligned}
\|h^{-1}\|_{C^{1+\beta}(U_{n+1}\setminus U_n)}&=\|h_n^{-1}\circ h_{n-1}^{-1}\circ\cdots\circ h_0^{-1}\|_{C^{1+\beta}(U_{n+1}\setminus U_n)} \\
&\le\prod_{i=0}^{n}\frac{r(n)^{1-\beta}}{r(n+1)}=\prod_{i=0}^{n}r(n)^{-1-\beta}. 
\end{aligned}
\end{equation}
Since $r(n-1) = r(n)^{\frac 12}$, we obtain that 
\begin{equation}\label{norm3}
\prod_{i=0}^{n}r(n)^{-1-\beta} = r(n)^{-(1+\beta)(1+\frac 12 + \frac 14 + \cdots + \frac 1{2^{n+1}})} = r(n)^{-2(1+\beta)(1-\frac 1{2^{n+2}})}.
\end{equation}
Finally, using \eqref{fP norm}, we find that
$$
\begin{aligned}
\|h\circ(f_{D^2}-&\text{Id})\circ h^{-1}||_{C^{1+\beta}(U_n\setminus U_{n-1})}\\
&\le r(n)^{-2\beta}r(n)^{2+2\kappa-\beta}r(n)^{-2(1+\beta)(1-\frac 1{2^{n+2}})} \\ 
&= r(n)^{2\kappa - 5\beta + \frac 1{2^{n+1}} -\frac{\beta}{2^{n+1}}}.
\end{aligned}
$$
One can choose $\beta$ such that $2\kappa - 5\beta + \frac 1{2^{n+1}} -\frac{\beta}{2^{n+1}} >0$ and conclude that   
$$
\|h\circ(f_{D^2}-\text{Id})\circ h^{-1}\|_{C^{1+\beta}(U_n\setminus U_{n-1})}\to 0
$$
as $n\to\infty$. Therefore, $f_P$ is tangent to $\text{Id}$ near $\partial P$ and hence, the map $f_P$ is of class $C^{1+\beta}$.

\textit{Proof of Theorem \ref{mapphi3}.} By construction, the map $f_P$ generates via a homeomorphism $T$ a $C^{1+\beta}$ diffeomorphism $f_M$ of the surface $M$. We construct a $C^\infty$ diffeomorphism $\psi: D^2\to D^2$ such that 
$\varphi_3:=T\circ h\circ\psi$ is the desired area preserving diffeomorphism (that is 
$(\varphi_3)_*m_{D^2}=m_M$) which can be continuously extended to the closure of $D^2$.

Denote $\mu=(h^{-1}\circ T^{-1})_* m_M$.  Since both $m_{D^2}$ and $m_M$ are normalized Lebesgue measures, we have
$$
\int_{D^2} d\,m_{D^2} =1=\int_M d\,m_M=\int_{D^2}\,d\mu.
$$
To obtain the desired result it suffices to show that there is a $C^\infty$ diffeomorphism 
$\psi: D^2\to D^2$ that can be continuously extended to $\partial D^2$ such that 
$\psi_* \mu=m_{D^2}$. 

Set $\mu_1=m_{D^2}$ and for $n>1$ define a sequence of measures $\mu_n$ such that 
\begin{enumerate}
\item[(i)] $\mu_n\in C^\infty(D^2)$ that is the measure $\mu_n$ is absolutely continuous with respect to $m_{D^2}$ with density function of class $C^\infty$;
\item[(ii)] $\mu_n=\mu$ on $h^{-1}(U_{n-1})$;  
\item[(iii)] $\int_{h^{-1}(U_n)} d\mu_n=\int_{h^{-1}(U_n)} d\mu$.
\end{enumerate}
It is clear that for any $n\ge 1$, $\int_{D^2} d\mu_n=\int_{D^2} d\mu=1$.

We need the following version of Moser's theorem (see \cite{GS79}, Lemma~1).

\begin{lem}\label{LVolume}
Let $\omega$ and $\mu$ be two volume forms on an oriented manifold $M$ and let $K$ be a connected compact set such that the support of $\omega-\mu$ is contained in the interior of $K$ and $\int_K d\omega= \int_K d\mu$. Then there is a $C^\infty$ diffeomorphism 
$\hat{\psi}: M\to M$ such that $\hat{\psi}|{(M\setminus K)}=\text{Id}|{(M\setminus K)}$ and 
$\hat{\psi}_* \omega=\mu$.
\end{lem}
Applying Lemma~\ref{LVolume} to each compact sets 
$K_n=h^{-1}(\overline{U}_n\setminus U_n)$ and volume forms 
$\mu_{n+1}|(\overline{U}_n\setminus U_n)$ and $\mu_n|(\overline{U}_n\setminus U_n)$, we obtain a $C^\infty$ diffeomorphism $\hat{\psi}_n: D^2\to D^2$ such that 
$(\hat{\psi}_n)_*\mu_{n+1}=\mu_{n}$ and $\hat{\psi}_n|{h^{-1}(U_{n-1})}=\text{Id}$. Then we let 
\[
\psi_n=\hat{\psi}_n\circ\dots\circ\hat{\psi}_1 \quad\text{and}\quad \psi=\lim_{n\to\infty}\psi_n.
\]
The construction gives 
$\hat{\psi}_n(h^{-1}(U_n\setminus U_n))=h^{-1}(U_n\setminus U_n)$. Recalling that $r(n)$ satisfies \eqref{rn} and using \eqref{norm1},  \eqref{norm2}, and \eqref{norm3}, we find that $\text{diam}\,\hat{\psi}^{-1}_n(U_n)\le Cd_n$ where $C>0$ is a constant and $d_n$ is a decreasing sequence of numbers such that 
$\sum_{n=1}^\infty d_n<\infty$. This implies that $d(x,\hat{\psi}_n (x))\le Cd_n$ for any $x\in D^2$. It follows that for any $x\in D^2$ and $n>j>0$,
\[
\begin{aligned}
d(\psi_j(x),\psi_n(x))&\le\sum_{i=j}^{n-1} d(\psi_i(x),\psi_{i+1}(x)) \\
&\le\sum_{i=j}^{n-1} d\big(\psi_i(x),\hat{\psi}_i(\psi_{i}(x))\big)\le C\sum_{i=j}^{n-1}d_i.
\end{aligned}
\]
This implies that the sequence $\psi_n$ is uniformly Cauchy and hence, $\psi$ is well defined and continuous on $D^2$. We can also get that $\psi: D^2\to D^2$ is a $C^\infty$ diffeomorphism.

By construction, we know that $(\psi_n)_*\mu_{n+1}=\mu_1=m_{D^2}$. Note that $D^2=\cup_{n\ge 1}h^{-1}(U_n)$. Hence, for any $x\in D^2$ there is $n>0$ and a neighborhood of $x$ on which $\mu_{n+i}=\mu_n$ for any $i>0$. It follows that 
$\psi_*\mu=(\psi_n)_*\mu_{n}=m_{D^2}$ on the neighborhood and hence, 
$\psi_*\mu=m_{D^2}$ on $D^2$.

\section{Completion of the proof of Theorem \ref{mainthm1}}

\subsection{Representing the map $f_M$ as a Young diffeomorphism} 
Consider a smooth compact connected oriented surface $M$ of genius $g\ge 0$ and the diffeomorphism $f_M:M\to M$ given by Statement 4 of Theorem \ref{mapphi3}. In this section we represent the map $f_M$ as a Young diffeomorphism.

\begin{prop}\label{$f_M$ is Young}
The map $f_M$ is a Young diffeomorphism. More precisely, there exists $Q>0$ and a collection of $s$-subsets that satisfy Conditions (Y1)-(Y6).
\end{prop}
\begin{proof} First note that we already know that the map $f_{T^2}$ is a Young diffeomorphism, so we can assume that the genius $g\ge 1$. Consider the collection of 
$s$-subsets $H(\Lambda_i^s)$ and the return time 
$\tau:\Lambda\to\mathbb{N}$ for the map $f_{\T^2}$ defined in Section \ref{torustower}. Define $\Delta_i^s:=\varphi_3(\varphi_2(\varphi_1(H(\Lambda_i^s))))$ with the return time on $M$ (again denoted by $\tau$) given by 
$\tau(\varphi_3(\varphi_2(\varphi_1(x)))=\tau(x),\ x\in \T^2$. Let 
$\Delta=\bigcup_{i}\Delta_i^s$. We claim that $f_M$ is a Young diffeomorphism with respect to the collection of $s$-subsets $\Delta_i^s$. 

To prove this we need to check Conditions (Y1)--(Y6). Since the maps $\varphi_i$, $i=1,2,3$ are homeomorphisms and (Y1) and (Y2) are satisfied for the map $f_{\T^2},$ then these conditions are also satisfied for $f_M$. In addition, (Y5) and (Y6) hold true for $f_M$ since the maps  $\varphi_i$, $i=1,2,3$ preserve the area. 

To show (Y3) and (Y4) observe that the element of the Markov partition $P$ in the Young tower representation for the map $f_{\T^2}$ is away from the critical points $x_i,\ i=1,2,3,4$. This implies that the map $\varphi_3\circ\varphi_2\circ\varphi_1$ is a smooth diffeomorphism from $P$ onto its image. Since the return time function $F_M=f_M^{\tau}$ is defined on $\Delta\subset P$, (Y4) follows as it hold true for the map $f_{\T^2}$ by Proposition \ref{f is Young}. Note that (Y3) holds for the map $f_{\T^2}$ for some constant $0<a<1$. By the estimate (32) in \cite{THK}, this constant may be chosen as small as we wish by making $Q$ large enough. Thus, we obtain (Y3) for the map $f_M$.

Finally, arguing similarly it is easy to show that the diffeomorphism $f_{S^2}$ of the sphere $S^2$ is a Young diffeomorphism. This completes the proof of the proposition.
\end{proof}	

\subsection{Lower and upper polynomial bounds on the decay of correlations} 

To establish a lower bound on the decay of correlations we need the following result from \cite{SZ}.
\begin{prop}\label{SZ} 
Assume that $(M,m,f)$ is a Young diffeomorphism for which the greatest common denominator of numbers $\{\tau_i\}$, $\text{gcd}\{\tau_i\}=1$ and for which $m(\tau>n)=\mathcal{O}(\frac{1}{n^{\nu}})$ for some $\nu>0$. Assume also that  for some $C>0$ and all $x,y\in \Delta_i^s$,
$$
d(f^j(x), f^j(y))\le C\max\{d(x,y), d(f^{\tau_i}(x), f^{\tau_i}(y))\}.
$$
Then for any $\sigma>0$ and $h_1, h_2\in C^{\rho}(M)$: 
\begin{enumerate}
\item $\text{Cor}_n(h_1,h_2)=O(\frac{1}{n^{\nu-1}})$.
\item There exists a nested sequence of sets $M_1\subset M_2\cdots \subset M$ such that if $h_1, h_2$ are supported in $M_k$ for some $k>0$ then
\begin{equation}\label{Gouzel}
\text{Cor}_n(h_1,h_2)=\sum_{n>N}^{\infty}m(\{x\colon\tau(x)>N\})\int_{M}h_1\,dm\int_{M}h_2\,dm+r_{\nu}(n),
\end{equation}
where $r_{\nu}(n)=\mathcal{O}(R_{\nu}(n))$ and 
$$
R_{\nu}(n)=
\begin{cases}
\frac{1}{n^{\nu}}&  \text{if } \nu>2,\\ 
\frac{\log n}{n^2} &  \text{if } \nu=2, \\ 
\frac{1}{n^{2\nu-2}}& \text{if } 1<\nu<2.
\end{cases}  
$$ 
Moreover, if $\int_{M}h_2=0$, then 
$\text{Cor}_n(h_1,h_2)=\mathcal{O}(\frac1{n^{\nu}})$.
\end{enumerate}
\end{prop}
We now verify the assumptions of Proposition \ref{SZ}. To prove that $\text{gcd}\{\tau_i\}=1$ we observe that the maps $\varphi_i,\ i=1,2,3$ and the map $H$ are homeomorphisms. Hence, it suffices to prove this for the linear map $A$. Since $A$ is Bernoulli, all powers of $A$ are ergodic. If $\text{gcd}\{\tilde{\tau}_i\}=d$ with $d\ne 1$, then the map $A^d$ would not be ergodic. Hence, $\text{gcd}\{\tilde{\tau}_i\}=1$. The requirement 
$m(\tau>n)=\mathcal{O}(\frac{1}{n^{\nu}})$ follows from Lemma \ref{upperbound} with 
$\nu=\gamma'-1$. 

To show that $d(f_M^j(x), f_M^j(y))\le K\max\{d(x,y), d(f_M^{\tau_i}(x), f_M^{\tau_i}(x))\}$ we observe that this is true for the map $f_{\T^2}^{\tau_i}$ of the torus which is smoothly conjugate to the map $f_M^{\tau_i}$. Therefore, by Proposition \ref{SZ}, we have
\begin{equation}\label{Gouz1}
\text{Cor}_n(h_1,h_2)=\sum_{N>n}^{\infty}m(\{x\colon\tau(x)>N\})\int_{M}h_1\,dm\int_{M}h_2\,dm+r_{\gamma'}(n),
\end{equation}
where $r_{\gamma'}(n)=\mathcal{O}(R_{\gamma'}(n))$ and  
$$
R_{\gamma'}(n)=
\begin{cases}
\frac{1}{n^{\gamma'-1}}& \text{if }\gamma'>3,\\ 
\frac{\log n}{n^2}& \text{if }\gamma'=3,\\ 
\frac1{n^{2\gamma'-4}}&\text{if }2<\gamma'<3.
\end{cases}
$$
By Lemmas \ref{lowerbound} and \ref{upperbound}, we have that  
$$
\frac{C_8}{n^{\gamma-1}}<m(\{x\in \Delta\colon\tau(x)>n\})<\frac{C_{11}}{n^{\gamma'-1}},
$$
where $\gamma$ and $\gamma'$ are defined by \eqref{gamma-gamma'}. Since the homeomorphisms $\varphi_1$, $\varphi_2$, and $\varphi_3$ are measure preserving we also have the same estimates for the map $f_M^{\tau}$.

To obtain a lower bound for correlations we consider separately the cases $\gamma'\ge 3$ and $2<\gamma'<3$.

Assume first that $\gamma'>3$, which is true if $\a<\frac16$. By assumption, 
$\int_{M}h_1\,dm\int_{M}h_2\,dm>0$ and Equation \ref{Gouz1} yields that
$$
\text{Cor}_n(h_1,h_2)>K_1\frac{1}{n^{\gamma-2}}-K_2\frac{1}{n^{\gamma'-1}}.
$$
Using definitions of $\gamma$ and $\gamma'$ (see \eqref{gamma-gamma'}) and choosing 
any $0<\mu<\frac 12$, one can show that $\gamma-2<\gamma'-1$ for all $0<\alpha<\frac16$.\footnote{One can use a computer assisted calculation to show that 
$\gamma-2<\gamma'-1$ for all $0<\a<0.42...$.} We conclude that for some $C>0$, 
$$
\text{Cor}_n(h_1,h_2)>C\frac{1}{n^{\gamma-2}}.
$$
Now, we consider the case when $\frac16<\a<\frac14$. This implies that $\gamma'>2$. Depending on the valse of $\mu$, we may have either $\gamma'>3$ or $\gamma'<3$ and we assume that latter (otherwise we are back to the previous case). With this assumption we have
$$
\text{Cor}_n(h_1,h_2)>K_1\frac{1}{n^{\gamma-2}}-K_2\frac{1}{n^{2\gamma'-4}}.
$$
Choosing again $0<\mu<\frac 12$, one can show that $\gamma-2<2\gamma'-4$ holds for all $0<\a<\frac14$.\footnote{Again a computer assisted calculation to show that $\gamma-2<2\gamma'-4$ holds for all $0<\a<0.36...$.} Thus we have the desired estimate
$$
\text{Cor}_n(h_1,h_2)>C\frac{1}{n^{\gamma-2}}
$$
for some $C>0$ and all $0<\a<\frac 14$.

In the case $\int_{M}h_2\,dm=0$ the desired result follows directly from the last statement of  Proposition \ref{SZ}.\\

An upper bound for the decay of correlations follows from the first statement of Proposition \ref{SZ}. So in our case we have for some $C'>0$
$$
|\text{Cor}_n(\hat{h_1},\hat{h_2})|<\frac{C'}{n^{\gamma'-2}}.
$$

\subsection{The Central Limit Theorem} By Theorem \ref{mainthm1}, for any H\"older continuous function $h$ with $\int h=0$  we have 
$\text{Cor}_n(h,h)=\mathcal{O}(\frac{1}{n^{\gamma'-1}})$. This implies that the correlation function is summable, when $\gamma'>2$ that is when $0<\a<\frac14$. The desired result now follows from \cite{liv} (see also \cite{SZ}, Theorem 3.1).

\subsection{The Large Deviation property}  We consider the Young tower $Y$ that 
represents the map $f_M$. For convenience the area $m_M$ on $M$ will be denoted by $m$. By Theorem \ref{upperbound}, we have that
$$
m(\{x\in \Lambda: \tau(x)>n\})<C_{11} \frac1{n^{\gamma'-1}},
$$
where $\gamma'=\frac1{2\a}+\frac{1-\mu}{2^{\a+2}}$. Hence for 
$0<\a<\frac 14$ Theorem 4.2 in \cite{melnic} applies yielding 
$$
m\Bigl(\Bigl |\frac{1}{n}\sum\limits_{i=0}^{n-1}h(f_M^i(x))-\int h\Bigr |>\varepsilon\Bigr)<C_{h,\delta}\varepsilon^{-2(\gamma'-2-\delta)}n^{-(\gamma'-2-\delta)},
$$
where $\gamma'=\frac{1}{2\a}+\frac{1-\mu}{2^{\a+2}}$. Moreover, for each 
$\delta>0$ the constant $C_{h,\delta}$ depends on the H\"older norm of $h$ continuously. 

To get a lower bound we need to check the conditions of Theorem 4.3 in \cite{melnic}. More precisely, for the set $\hat{Y}_k=\{(x,l)\in \hat{Y}: \tau(x)>k\}$ it must be true that for some $k,$ $m(\pi(\hat{Y}_k))<1,$ where $\pi:\hat{Y}\to M$ is given by $\pi(x,k)=f^k(x)$ as before.

Given $k$ let us chose a partition element $\Delta_i$ of the base $\Delta$ of the tower for $f_M$ with $\tau(\Delta_i)\le k$. Then 
$\Delta_i\subset \hat{Y}\setminus\hat{Y}_k$ and obviously 
$\hat{m}(\Delta_i)>0.$ Thus $\hat{m}(\hat{Y}_k)<1$ and we obtain $m(\pi(\hat{Y}_k))<1$ as $\pi$ is measure preserving. Thus, by Theorem 4.3 in \cite{melnic}, we obtain the lower bound
$$
\frac{1}{n^{\gamma'-2+\delta}}<m\Bigl(\Bigl|\frac{1}{n}\sum\limits_{i=0}^{n-1}h(f_M^i(x))-\int h\Bigr|>\varepsilon\Bigr)
$$
for small $\varepsilon$, open and dense subset of H\"older continuous observables $h$, and infinitely many $n$.

\subsection{The measure of maximal entropy (MME)} Recall that the diffeomorphism $f_M$ of the surface $M$ is a Young diffeomorphism and consider the corresponding collection $\{\Delta_i^s\}$ of $s$-sets. Denote by 
$\mathcal{S}_n=\{\Delta_i^s: \tau(\Delta_i^s)=n\}$. Since the map $f_M$ is topologically conjugate to the toral automorphism $A$, the number 
$\mathcal{S}_n$ for $f_M$ is equal to the number $\mathcal{S}_n$ for $A$. The latter is known to satisfy $\mathcal{S}_n\le e^{hn}$ with $h<h_{\text{top}}(A)$ (see \cite{THK}). It now follows from \cite{ind} (see Theorem 7.1) and \cite{SZ} that the map $f_M$ possesses a unique MME which has all the desired properties.

\end{document}